\documentclass[11pt]{article}

\usepackage[utf8]{inputenc}
\usepackage[colorlinks=true,linkcolor=black,citecolor=black]{hyperref}
\usepackage{geometry}
\usepackage{graphicx}
\usepackage{algorithm,algcompatible}
\usepackage{booktabs} 
\usepackage{array} 
\usepackage{paralist} 
\usepackage{verbatim} 
\usepackage{subfig} 
\usepackage{fancyhdr} 
\usepackage[nottoc,notlof,notlot]{tocbibind}
\usepackage[titles,subfigure]{tocloft}

\usepackage{amsmath,amsfonts,amsthm,mathrsfs,amssymb,cite}
\usepackage[usenames]{color}

\numberwithin{equation}{section}

\theoremstyle{theorem}
\newtheorem{Lem}{Lemma}
\newtheorem{theorem}{Theorem}

\newtheorem{prop}[theorem]{Proposition}

\newtheorem{remark}{Remark}

\newcommand{\bitem}{\begin{itemize}}
	\newcommand{\eitem}{\end{itemize}}

\newcommand{\R}{\mathbb{R}}
\newcommand{\bpm}{\begin{pmatrix}}
	\newcommand{\epm}{\end{pmatrix}}

\newcommand{\TV}{{\rm TV}}
\newcommand{\bq}{\begin{equation}}
\newcommand{\eq}{\end{equation}}

\newcommand{\D}{\mathrm{div} \;}

\let\abs=\envert

\let\norm=\enVert

\let\inprod=\inProd

\title{\bf Efficient ADMM and Splitting Methods for Continuous Min-cut and Max-flow Problems }
\author{Hongpeng Sun\thanks{Institute for Mathematical Sciences,
		Renmin University of China, 100872 Beijing, People's Republic of China.
		Email: \href{mailto:hpsun@amss.ac.cn}{hpsun@amss.ac.cn}.}
	   \and Xuecheng Tai\thanks{Department of Mathematics, Hong Kong Baptist University,
    			Kowloon Tong, Hong Kong. Email: \href{mailto:xuechengtai@hkbu.edu.hk}{xuechengtai@hkbu.edu.hk}}.
       \and Jing Yuan\thanks{School of Mathematics and Statistics, Xidian University. Email: \href{mailto:jyuan@xidian.edu.cn}{jyuan@xidian.edu.cn}}. }

\begin{document}

	\maketitle
	
	\begin{abstract}
		The Potts model has many applications. It is equivalent to some min-cut and max-flow models. Primal-dual algorithms have been used to solve these problems. Due to the special structure of the models, convergence proof is still a difficult problem. In this work, we developed two novel, preconditioned, and over-relaxed alternating direction methods of multipliers (ADMM) with convergence guarantee for these models.  Using the proposed preconditioners or block preconditioners, we get accelerations with the over-relaxation variants of preconditioned ADMM. The preconditioned and over-relaxed Douglas-Rachford splitting methods are also considered for the Potts model.   Our framework can handle both the two-labeling or multi-labeling problems with appropriate block preconditioners based on  Eckstein-Bertsekas and Fortin-Glowinski splitting techniques.  
	
	\end{abstract}


	\section{Introduction}
	\label{intro}
	
	During recent twenty years, convex optimization was successfully introduced as a powerful tool to image processing, computer vision and machine learning, which is mainly credited to the pioneering works from both theoretical and algorithmic studies \cite{Nesterov2005Smooth,citeulike3001108,Beck2009A,ChambolleP11,He2002A,Nikolova2006,Yuan2010}, along with vast applications,
	for examples, the total-variation-based image denoising~\cite{rudin1992nonlinear,goldstein2014fast}, image segmentation \cite{Nikolova2006,iwr_08,ChambolleP11,Yuan2010}, the sparsity-based image reconstruction~\cite{Beck2009A,zhang2010analysis}, 
	and total-variation-based motion estimation \cite{yuan2007discrete} etc.
	
	The basic convex optimization theory related to these applications aims to minimize a
	finite sum of convex function terms:
	\bq \label{eq:pmodel}
	\min_u \;\; f_1(u) \, + \, \ldots \, + \, f_n(u) \; ,
	\eq
	where it also models the convex constrained optimization problem as its special case, such that 
	the convex constraint set $C$ on the variables $u(x) \in C$ 
	can be rewritten by adding the associate characteristic function 
	into \eqref{eq:pmodel}.
	
	Given the very high dimension of the solution $u$ of many applications, the iterative first-order optimization schemes, which essentially make use of the first-order gradient information, play the central role in building up practical algorithmic implementations with an affordable computational cost per iteration.
	In this perspective, the conjugate or duality form of each convex function term in \eqref{eq:pmodel}
	\bq \label{eq:conj}
	f_i(u) \, = \, \max_{p_i} \, \inprod{u,p_i} -f_i^*(p_i)
	\eq
	provides one of the most powerful tool in both analyzing and developing such first-order 
	iterative algorithms, for which the introduced new dual variable $p_i$ for each
	functional term $f_i$ just represents the first-order gradient of $f_i(u)$
	implicitly. By simple computations, this brings two equivalent optimization models to the studied convex minimization problem \eqref{eq:pmodel}, a.k.a.
	the \emph{primal-dual model}:
	\bq \label{eq:pdmodel}
	\min_u \max_p \; \; \underbrace{\inprod{p_1 + \ldots + p_n, u} \, - \,
		f_1^*(p_1) \, - \, \ldots \, - \, f_n^*(p_n)}_{\text{Lagrangian function } L(u,
		p)},\; \ \ p: =(p_1, \cdots, p_n),
	\eq
	and the \emph{dual model} 
	\bq \label{eq:dmodel}
	\left. \begin{array}{ll} \max_p \; \; &\,  - \, f_1^*(p_1) \, - \, \ldots
		\, - \, f_n^*(p_n)\,, 
		\\ \text{s.t.  } \,  & \, p_1 + \ldots + p_n \, = \, 0.
	\end{array}
	\right. \, 
	\eq 
	
	Actually, for each convex function $f_i(u)$, the optimum of $p_i$ for its dual expression \eqref{eq:conj} is nothing but
	its corresponding gradient or subgradient at $u$; therefore, the linear equality constraint $p_1 + .. + p_n  = 0$ for the dual model
	\eqref{eq:dmodel} exactly represents the first-order optimal condition to the studied convex optimization problem \eqref{eq:pmodel}, i.e.,
	\[
	0 \in \partial f_1(u) \, + \,  \ldots \, + \, \partial f_n(u).
	\]
	In addition, for the dual optimization problem \eqref{eq:dmodel}, the optimum multiplier $u^*$ to its linear equality constraint $p_1 + .. + p_n  = 0$ is just the minimum of the original convex optimization problem \eqref{eq:pmodel}, which can be easily seen by the formulation \eqref{eq:pdmodel}.
	
	Especially, each term $f_i^*(p_i)$, $i=1 \ldots n$, in the energy functional of 
	the dual model \eqref{eq:dmodel} solely depends on an independent variable $p_i$ which is loosely correlated to the
	other variables by the linear equality constraint $p_1 + p_2\cdots+ p_n  = 0$. This is in contrast to
	its original optimization model \eqref{eq:pmodel} whose energy functional terms are interacted with each with the
	common unknown variable $u$. This provides a big advantage in develop splitting optimization algorithms,
	to tackle the underlying convex optimization problem, particularly at a large scale. For instance, 
	the classical augmented Lagrangian method (ALM)~\cite{citeulike1859441,Rockafellar1976} provides an optimization framework to develop the corresponding algorithmic scheme for the linearly
	constrained dual model \eqref{eq:dmodel}, which involves two
	sequential steps at each iteration:
	\begin{align} \label{eq:ALM}
	p^{k+1} \, := & \, \arg\max_p \, L(u^k,p) \, - \, \frac{c^k}{2}\norm{p_1
		+ \ldots + p_n}^2 \, , \\
	u^{k+1} \, = & \, u^k \, -
	\, c^k(p_1^{k+1} + \ldots + p_n^{k+1}) \, ,
	\end{align}
	where the positive parameter $c^k > 0$ is the associate step-size. 
	
	In this work, we focus on developing novel efficient convex optimization methods, based on primal-dual optimization theory, to image segmentation, As one of the most fundamental problems of image processing and computer vision, a lot of contributions were dedicated into image segmentation during last three decades. Despite big progresses upon current convolution neural networks (CNNs), whose results heavily rely on both the quantity and quality of training data, one of the most successful and popular mathematical models for image segmentation is firmly rooted in the theory of min-cut/max-flow, which was originally developed on account of Markov random fields (MRF), and a series efficient and robust solvers have been developed based on graph-cuts \cite{Boykov01fastapproximate}.
	%
	In \cite{BYT,Yuan2010,ybtb2014}, continuous max-flow and min-cut problem is considered. It is shown that the Potts model is equivalent to a continuous min-cut and max-flow problem. If one discrete these models with specific approximations and the so-called "length" term, they reduce to existing graph-cut models. However, they can also produce some discrete models that are not submodulus and can be solved with fast algorithms.   Augmented Lagrangian method was used  to solve these problems \cite{ybtb2014}. For augmented Lagrangian methods, one needs to solve all the dual variables simultaneously followed by updating the Lagrangian multipliers, which is difficult for practical applications. Thus, alternating direction method of multipliers (ADMM) are actually employed in \cite{Yuan2010,ybtb2014} for solving the dual problem. However, although the dual variables can be solved with ADMM consecutively, which are more convenient than augmented Lagrangian method, the convergence can not be guaranteed due to the three dual variables. It is recently found that ADMM iterations can be divergent if there are more than two block of variables \cite{CBHY}, which is the case in \cite{Yuan2010,ybtb2014} unfortunately.

	Our contributions are two-fold.  We first propose two convergent ADMM methods to solve these models. The first step is regrouping the dual variables into two big blocks. However, some multiple variables are still coupled with each other and the subproblems are still hard to solve.  In order to solve the nonlinear subproblems, we proposed delicate preconditioners for easy computations of the subproblems. Various efficient preconditioners or block-preconditioners are proposed for both the two-phase or multi-phase problems. 
	To the best knowledge of the authors, the ADMM proposed in this paper are the first convergent variants of ADMM for continuous max-flow problems in the literature.
	
	Our second contribution is that we can get accelerations with some over-relaxed variants of the ADMM with preconditioners. It is shown that for various optimization and regularization problems including the image denoising problems, over-relaxation can gain accelerations \cite{DY,LST,SUN}. There are two kinds of over-relaxations for ADMM including the over-relaxation originated from \cite{DY,FG,LST} and the over-relaxation from \cite{EP}. Both kinds of over-relaxations are considered for two-phases and multi-phase  problems. As shown in the numerical parts, good accelerations can be obtained with our over-relaxation and preconditioning technique. We also considered the preconditioned Douglas-Rachford splitting algorithm for comparison.
	
	The paper is organized as follows. In section \ref{sec:intro}, we give an introduction of the min-cut primal approach and max-flow dual approach for both two-phase and multi-phase problems. We also give a brief introduction of some existing algorithms.  In section \ref{sec:two:labels}, we focus on two novel variants of convergent relaxed and preconditioned ADMM, along with the
	classical ADMM in \cite{BYT,YBTB10,ybtb2014} for both two-phase and multi-phase image  segmentation problems.
	Preconditioners or block-preconditioners are developed for the corresponding algorithms.
	The relaxed and preconditioned Douglas-Rachford splitting method for the saddle-point approach is also discussed.  In section \ref{sec:num:experi}, we present some detailed numerical experiments to show the efficiency of the proposed algorithms. In the section \ref{sec:conclusion}, we give a some conclusions and discussions.

	\section{Preliminary Theories}
	\label{sec:intro}
	
	

	\subsection{Primal and Dual Models with continuous min-cut and max-flow settings}
	
	In this study, we focus on the classical convex optimization models to image segmentation, which are essentially proposed and formulated based on the theory of Markov random fields (MRF).
	During the last decades, the min-cut model was becoming  one of the most successful models for 
	foreground-background image segmentation~\cite{Boykov01anexperimental,Boykov01fastapproximate}, which has been well studied in the discrete 
	graph setting and can be efficiently solved by the scheme for max-flow problems. In fact, such min-cut model can be
	also formulated in a spatially continuous setting, i.e.
	the \emph{spatially continuous min-cut problem}~\cite{Nikolova2006}:
	\bq \label{seg:chanvese:bi}
	\min_{u(x) \in \{0,1\}} \;
	\int_{\Omega} \big\{ \big(1-u \big)C_t  + u \, C_s \big\}(x) \, dx + \alpha \int_{\Omega}\abs{\nabla u}\, dx \, ,
	\eq
	where $C_s(x)$ and $C_t(x)$ are the cost functionals such that, for each pixel $x \in \Omega$, $C_s(x)$ and $C_t(x)$ give the costs to label $x$ as 'foreground' and 'background' respectively. The optimum $u^*(x)$ to the combinatorial optimization
	problem~\eqref{seg:chanvese:bi} defines the optimal foreground segmentation region $S$ such that $u^*(x)=1$ for any $x \in S$, and the background segmentation region $\Omega\backslash S$ otherwise. 
	
	Chan et al.~\cite{Nikolova2006} proved that the challenging non-convex combinatorial optimization problem~\eqref{seg:chanvese:bi}
	can be solved globally by computing its convex relaxation model, i.e. the \emph{convex relaxed min-cut model}:
	\bq \label{seg:chanvese}
	\min_{u(x) \in [0,1]} \;
	\int_{\Omega} \big\{ \big(1-u \big)C_s  + u \, C_t \big\} \, dx + \alpha \int_{\Omega}\abs{\nabla u}\, dx \, ,
	\eq
	while thresholding the optimum of \eqref{seg:chanvese} with any parameter $\beta \in (0,1)$. Hence, the difficult combinatorial optimization problem \eqref{seg:chanvese:bi} can be exactly solved by a convex minimization problem \eqref{seg:chanvese} instead.
	Particularly, Yuan et al.~\cite{Yuan2010,ybtb2014} observed that the convex relaxed min-cut model \eqref{seg:chanvese} can be equivalently reformulated
	by its dual model, i.e. the \emph{continuous max-flow model}:
	\begin{align}
	\max_{p_s, p_t, q} \; & \int_{\Omega} p_s(x) \, dx \,  \label{eq:two-label-mf} \\
	\textbf{s.t.} \;\;& \abs{q(x)}   \leq \alpha \, ,  \quad 
	p_s(x) \leq  C_s(x) \, , \quad
	p_t(x) \leq  C_t(x) \,  ;  \label{eq:mpf-two-cond-01}\\
	& \Big(\D q - p_t + p_s \Big)(x)
	\, = \, 0 \,  \,  . \label{eq:mpf-two-cond-03}
	\end{align}
	
	For multiphase image segmentation, \emph{Potts model} is used as the basis to formulate the associate 
	mathematical model~\cite{Boykov01fastapproximate,Boykov01anexperimental} by minimizing the following energy function
	\bq \label{eq:potts-model} 
	\min_{u} \, \sum_{i=1}^n
	\int_{\Omega} u_i(x) \, \rho(l_i, x) \, dx \, + \, \alpha
	\sum_{i=1}^n \int_{\Omega} \abs{\nabla u_i} \, dx \, 
	\eq
	subject to
	\bq \label{eq:potts-bconst}
	\sum_{i=1}^n \, u_i(x) \, = \, 1 \, , \quad { u_i(x) \in
		\{0,1\}} \,, \; i = 1 \ldots n\,  , \quad \forall x \in \Omega \, ,
	\eq
	where $\rho(l_i, x)$, $i=1 \ldots n$, are the cost functionals: for each pixel $x \in \Omega$, $\rho(l_i, x)$ gives
	the cost to label $x$ as the segmentation region $i$. Potts model seeks the optimum labeling function $u_i^*(x)$, $i=1 \ldots n$, 
	to the combinatorial optimization
	problem~\eqref{eq:potts-model}, which defines the segmentation region $\Omega_i$ such that $u_i^*(x)=1$ for any $x \in \Omega_i$.
	Clearly, the linear equality constraint $u_1(x) + \ldots + u_n(x) = 1$ states that each pixel $x$ belongs to a single segmentation region. 
	
	Similar as the convex relaxed min-cut model \eqref{seg:chanvese}, 
	we can relax each binary constraint $u_i(x) \in \{0,1\}$ in \eqref{eq:potts-bconst} to 
	the convex set $u_i(x) \in [0,1]$, then formulate the
	convex relaxed optimization problem of Potts model \eqref{eq:potts-model} as
	\bq \label{eq:relaxed-potts} 
	\min_{u \in S} \, \sum_{i=1}^n
	\int_{\Omega} u_i(x) \, \rho(l_i, x) \, dx \, + \, \alpha
	\sum_{i=1}^n \int_{\Omega} \abs{\nabla u_i} \, dx \, 
	\eq
	where
	\bq \label{eq:pw-simplex}
	S \, = \, \{u(x) \,|\, (u_1(x) , \ldots, u_n(x)) \, \in \,
	\triangle_n^{+} \, , \; \forall x \in \Omega \, \} \, .
	\eq
	$\triangle_n^{+}$ is the simplex set in the space $\R^n$. 
	
	
	Through variational analysis (c.f.~\cite{YBTB10}), it was proven that the  dual formulation, 
	i.e. the following \emph{continuous max-flow model},  is equivalent to the convex relaxed Potts model~\eqref{eq:relaxed-potts}:
	\bq \label{eq:potts-mf}
	\max_{p_s, p, q} \;  \int_{\Omega} p_s \,
	dx,  
	\eq
	subject to
	\bq
	\label{eq:mpf-cond-01} \abs{q_i(x)} \, \leq \, \alpha \, , \quad
	p_i(x) \, \leq \, \rho(\ell_i,x) \, , \quad i=1 \ldots n \, ; \eq
	\bq \label{eq:mpf-cond-03} \big( \D q_i - p_s + p_i \big)(x) \,
	= \, 0 \, , \quad i=1,\ldots,n \, . 
	\eq 
	
	The linear equality constraints \eqref{eq:mpf-two-cond-03} and \eqref{eq:mpf-cond-03} just correspond to the classical flow conservation conditions of the max-flow models \eqref{eq:two-label-mf} and \eqref{eq:potts-mf} respectively.
	
	\subsection{ALM-Based Splitting Algorithms}
	In this section, we will give a brief review of the classical ALM-based algorithms for solving both foreground-background and multiphase image segmentation problems \cite{Yuan2010,ybtb2014,YBTB10}.  These were developed under the perspective of dual formulations \eqref{eq:two-label-mf} and \eqref{eq:potts-mf}. 
	For example, the dual formulation of \eqref{eq:two-label-mf} is as follows: 
	\begin{equation}\label{eq:seg:dual}
	\max_{p_s, p_t, q}\; \langle 1, p_s \rangle  -{I}_{\{p_s \leq C_s\}}(p_s) - {I}_{\{p_t \leq C_t\}}(p_t) - {I}_{\{\|q\|_{\infty} \leq \alpha\}}(q),
	\end{equation}
	subject to the linear equality constraint \eqref{eq:mpf-two-cond-03}, i.e. the flow-conservation condition.
	The convex set constraints of \eqref{eq:mpf-two-cond-01} are encoded in the energy functional of \eqref{eq:seg:dual} by their associate indicator functions ${I}_{\{p_s \leq C_s\}}(p_s)$, ${I}_{\{p_t \leq C_t\}}(p_t) $ and ${I}_{\{\|q\|_{\infty} \leq \alpha\}}(q)$ respectively. 
	Given its corresponding augmented Lagrangian functional:
	\begin{align}
	L_c(u, p_s, p_t, q) =& \langle 1, p_s \rangle  -{I}_{\{p_s \leq C_s\}}(p_s) - {I}_{\{p_t \leq C_t\}}(p_t) - {I}_{\{\|q\|_{\infty} \leq \alpha\}}(q) \notag \\
	&+ \langle u, p_t - p_s + \D q\rangle - \frac{c}{2}\|p_t - p_s + \D q\|_{2}^2 \, , \label{eq:augfunc:twolabel}
	\end{align}
	in view of the typical ALM scheme \eqref{eq:ALM}, optimizing all the dual variables, e.g. $(p_s, p_t, q)$ for \eqref{eq:two-label-mf} and $(p_s, p_i, q_i)$ for \eqref{eq:potts-mf}, simultaneously at each iteration is impractical. Henceforth, the parameter $c$ denotes the step size of ALM scheme. 
	ADMM, by optimizing each dual variable block sequentially, is thus employed in \cite{Yuan2010,ybtb2014,YBTB10} for actual implementations.
	However, it turns out that solving the nonlinear subproblem involving with $q$ is still challenging, i.e.,
	\[
	\left(\partial {I}_{\{\|q\|_{\infty} \leq \alpha\}}(\cdot)  + c \D^* \D \right)^{-1}.
	\]
	In \cite{Yuan2010,ybtb2014,YBTB10}, the one-step projection is introduced for the update of $q^{k+1}$ as follows along with the updates of all the remaining variables, i.e., 
	\begin{equation}\label{eq:ADMM:Tai_yuan} 
	\begin{cases}
	q^{k+1} = \mathcal{P}_{\alpha}\left( (I-\frac{1}{a}\nabla \nabla ^*)q^k  + \frac{1}{a} \nabla(p^k_t -p^k_s -\frac{u^k}{c})\right), \\\
	p_{s}^{k+1} = \mathcal{P}_{C_s}(p_{t}^k + \D q^{k+1} -\frac{u^k}{c} + \frac{1}{c}),  \\
	p_{t}^{k+1} = \mathcal{P}_{C_t}(p_{s}^{k+1} - \D q^{k+1}+ \frac{u^k}{c}), \\
	u^k = u^k - c(p_t^{k+1} - p_s^{k+1} + \D q^{k+1}) \, ,
	\end{cases}
	\end{equation}
	where the projections $\mathcal{P}_{\alpha}$, $\mathcal{P}_{C_s}$ and $ \mathcal{P}_{C_t}$ are as follows:
	\begin{equation}\label{eq:projections}
	\mathcal{P}_{\alpha}(\tilde q) = \frac{\tilde q}{\max(1.0, \frac{|\tilde q|}{\alpha})},\ \ \mathcal{P}_{C_s}(\tilde p_s) = \min(\tilde p_s, C_s), \ \  \mathcal{P}_{C_t}(\tilde p_t) = \min(\tilde p_t, C_t) \,.
	\end{equation}
	However, the convergence of \eqref{eq:ADMM:Tai_yuan} is not guaranteed, where it is recently discovered that such an ADMM iteration could be divergent for optimizing more than two blocks of variables consecutively \cite{CBHY}, which is exactly the cases studied in \cite{Yuan2010,ybtb2014,YBTB10} unfortunately. This motivates us to design convergent and more efficient ADMM-based method.

	\section{Novel ADMM-Based Optimization Methods}
	\label{sec:two:labels}

	In this section, we propose two novel convergent ADMM types of optimization algorithms, i.e. Eckstein-Bertsekas-type and Fortin-Glowinski-type for solving the studied image segmentation problems in terms of \eqref{eq:two-label-mf} and \eqref{eq:potts-mf}. Particularly, our experiments show the derived algorithms both outperform the classical ALM-based algorithms introduced in \cite{Yuan2010,ybtb2014,YBTB10}.
	
	\subsection{Novel ADMMs for Foreground-Background Image Segmentation}

	\subsubsection{Relaxed and Preconditioned ADMM of Eckstein-Bertsekas Type}
	
	Let the new variable $p$ denote the variable blocks $(p_s, p_t)$. We can equivalently 
	generalize the dual model \eqref{eq:two-label-mf} of foreground-background image segmentation as the optimization of two variable blocks $(p,q)$ such that
	\bq \label{eq:two-label-mf:twoblocks}
	\max_{p, q} \; -G(p) - H(q),
	\eq
	subject to
	\bq
	Ap \, + \, Bq \, =\, 0 \, ,
	\eq
	where
	\[
	G(p) \, := \,  \langle 1, p_s \rangle  - {I}_{\{p_s \leq C_s\}}(p_s) - {I}_{\{p_t \leq C_t\}}(p_t), \ \ H(q):={I}_{\{\|q\|_{\infty} \leq \alpha\}}(q),
	\]
	and 
	\[
	A \,= \, [I,-I] \, ,  \quad B \, = \, \D \, .
	\]
	Its associated augmented Lagrangian functional can thus be formulated as follows:
	\bq \label{eq:alm:II}
	L(u, p, q) \, =\, -G(p) -H(q)  +  \langle u, Ap+Bq \rangle - \frac{c}{2}\|Ap+Bq\|_{2}^2 \,.
	\eq
	
	For the given matrix 
	\begin{equation}\label{eq:AadA}
	A^*A = \begin{bmatrix}
	I &  -I \\-I &I 
	\end{bmatrix},
	\end{equation}
	we see that $cA^*A +\partial G $ and $cB^*B + \partial H$ are nonlinear and maximal monotone operators, which, however, have no explicit inverse. The subproblems for ADMM involving $p$ and $q$ thus are very challenging to solve.

	Now, let's turn to the relaxed and preconditioned ADMM of Eckstein-Bertsekas type~\cite{SUN},
	which actually origins from the relaxed Douglas-Rachford splitting method to the dual problem \cite{EP}.
	The relaxed and preconditioned ADMM of Eckstein-Bertsekas type for solving the equivalent dual problem \eqref{eq:two-label-mf:twoblocks} reads as follows \cite{SUN},
	\begin{align}
	q^{k+1} &=(N + \partial H)^{-1}(B^*(-cAp^k + u^k) + (N-cB^*B)q^k), \notag \\
	p^{k+1} &= (M + \partial G)^{-1}(A^*(-c\rho_kBq^{k+1} +c(1-\rho_k)Ap^k + u^k)+(M-cA^*A)p^k), \notag\\
	u^{k+1} &=u^k - c(Ap^{k+1}-(1-\rho_k)Ap^k+\rho_kBq^{k+1})\, ,\label{eq:update:two:sun:lag}  
	\end{align}
	where $\{\rho_k \in (0,2)\}$ is a non-decreasing sequence, $c$ is the step size as before and $N$, $M$ are two bounded and linear operators (or matrices) satisfying
	\begin{equation}\label{eq:con:sun:pradmm}
	N-cB^*B \geq 0, \quad M-cA^*A \geq 0,
	\end{equation}
	which are sufficient for the convergence of \eqref{eq:update:two:sun:lag}. However, designing $M$ and $N$ satisfying \eqref{eq:con:sun:pradmm} such that $M +\partial H$ and $N + \partial G$ are more efficient to invert is very challenging. We present several different strategies depending on the corresponding operators.  For \eqref{eq:update:two:sun:lag}, we choose,
	\[
	N = acI, \quad M = \tilde a cI_2, \quad I_2 : = \text{Diag}[I,I].
	\]
	For the choice of the operators $M$ and $N$ and the convergence of  \eqref{eq:update:two:sun:lag}, we introduce the preconditioners to both dual variables with mild conditions and we have the following theorem. 
	
	\begin{theorem}\label{thm:twolabel}
		For the discretized divergence operator $\D$ and matrix $A = [I,-I]$, we have
		\[
		8I \geq \D^* \D, \quad 2I_2 \geq A^*A.
		\]
		We thus can choose $M=2cI$ and $N=8cI_2$ with $a=8$ and $\tilde a=2$ satisfying the condition \eqref{eq:con:sun:pradmm}. Assuming $\{\rho_k \in (0,2)\}$ is a non-decreasing sequence, then $(q^k,p^k, u^k)$ converges weakly
		to a saddle-point $(q^*, p^*, u^*)$  of \eqref{eq:alm:II} and $(q^*, p^*)$ is a solution of \eqref{eq:two-label-mf:twoblocks}.
	\end{theorem}
	\begin{proof}
		It is known that $\D^*\D \leq 8I$, c.f.~\cite{CP}. Since for any $v = [x,y]^{T}$, we have
		\[
		[x,y]\begin{bmatrix}
		2I& 0 \\
		0& 2I
		\end{bmatrix}\begin{bmatrix}
		x \\y
		\end{bmatrix}
		=2x^2 + 2y^2 \geq x^2 + y^2 -2xy =v^TA^*Av =  [x,y]\begin{bmatrix}
		I& -I \\
		-I& I
		\end{bmatrix}\begin{bmatrix}
		x \\y
		\end{bmatrix},
		\]
		we get $2I_2 \geq A^*A$. The remaining convergence follows Theorem 4.1 in \cite{SUN}.
	\end{proof}
	
	%
	%
	By simple computation, we further write the detailed steps for each iteration of \eqref{eq:update:two:sun:lag} applying the \emph{continuous max-flow model} \eqref{eq:mpf-two-cond-01}: 
	\begin{align}
	q^{k+1}& = \mathcal{P}_{\alpha}\left( (I-\frac{1}{a}{\D}^{*}\D)q^k  + \frac{1}{a} \nabla(p^k_t -p^k_s -\frac{u^k}{c})\right), \notag \\
	p_{t}^{k+1}&= \mathcal{P}_{C_t}\left(p_{t}^k  - \frac{\rho_k}{\tilde a}(p_t^k-p_s^k) +\frac{1}{\tilde a}(-\rho_k\D q^{k+1} + \frac{1}{c}u^k)\right),  \label{eq:pre:ADMM:twolabels:SUN} \tag{rpADMMII} \\
	p_{s}^{k+1}&=\mathcal{P}_{C_s}\left( p_s^k - \frac{\rho_k}{\tilde a}(p_s^k-p_t^k) +  \frac{1}{\tilde a}(\rho_k\D q^{k+1} - \frac{1}{c}u^k +\frac{1}{c})\right), \notag \\
	u^{k+1} &=u^k - c\left((p_t^{k+1}-p_s^{k+1})-(1-\rho_k)(p_t^{k}-p_s^k)+\rho_k \D q^{k+1}\right). \notag
	\end{align}
	For the above algorithm \ref{eq:pre:ADMM:twolabels:SUN}, the projections $\mathcal{P}_{\alpha}$, $\mathcal{P}_{C_t}$, $\mathcal{P}_{C_s}$ are the same as in \eqref{eq:projections} and the parameters $a$, $\tilde a$ given in Theorem \ref{thm:twolabel}. $\rho_k \equiv 1.9$ is preferred in numerical computation. This kind of relaxation is originated from \cite{EP} and the principle is different from  the relaxation in \cite{FG}.
	\subsubsection{Relaxed and Preconditioned ADMM of Fortin-Glowinski Type}
	Now, let's turn to another novel relaxed preconditioned ADMM of Fortin-Glowinksi type for \eqref{eq:alm:II} as follows (c.f. \cite{DY,FG,LST}):
	\begin{equation}\label{eq:update:two:tai:lag} 
	\begin{cases}
	q^{k+1} = \text{argmax}_{q}L(u^k;p^k, q) - \frac{1}{2}\|q-q^k\|_{P},\\
	p^{k+1} = \text{argmax}_{p}L(u^k;p, q^{k+1}) -\frac{1}{2}\|p-p^k\|_{Q},\\
	u^{k+1} =u^k - rc(Ap^{k+1}+Bq^{k+1}),
	\end{cases}
	\end{equation}
	where the two linear operators $P := acI-c\D^*\D$ and $Q:=\tilde acI-cA^*A$. Unlike relaxation using \eqref{eq:update:two:sun:lag}, there is only relaxation on the updates of the Lagrangian multiplier $u$. 
	
	The linear operators $P$ and $Q$ are required to be positive semi-definite for the convergence with relaxation parameter $r \in (0, \frac{\sqrt{5}+1}{2})$, c.f. \cite{DY,LST,FG}. Theorem \ref{thm:twolabel} reveals that we need  $a \geq 8$ and $\tilde a \geq 2$. 
	%
	By simple computation, we get the following detailed steps for each iteration of \eqref{eq:update:two:tai:lag}:
	\begin{align}\label{eq:pre:ADMM:twolabels:FG} \tag{rpADMMI}
	\begin{cases}
	q^{k+1} = \mathcal{P}_{\alpha}\left( (I-\frac{1}{a}{\D}^{*}\D)q^k  + \frac{1}{a} \nabla(p^k_t -p^k_s -\frac{u^k}{c})\right),  \notag\\
	p_{t}^{k+1}= \mathcal{P}_{C_t}\left(p_{t}^k  - \frac{1}{\tilde a}(p_t^k-p_s^k) +\frac{1}{\tilde a}(-\D q^{k+1} + \frac{1}{c}u^k)\right), \\
	p_{s}^{k+1}=\mathcal{P}_{C_s}\left( p_s^k - \frac{1}{\tilde a}(p_s^k-p_t^k) +  \frac{1}{\tilde a}(\D q^{k+1} - \frac{1}{c}u^k +\frac{1}{c})\right), \notag\\
	u^{k+1} = u^k - rc\left((p_t^{k+1}-p_s^{k+1})+\D q^{k+1}\right). \notag
	\end{cases}
	\end{align}
	In the sequel, we choose $a=8$, $\tilde a = 2$ and the relaxation parameter $r=1.618$ for the experiments by \eqref{eq:pre:ADMM:twolabels:FG}. 
	
	\begin{remark}\label{remark:two}
		The ADMM  \eqref{eq:ADMM:Tai_yuan} in \cite{Yuan2010,ybtb2014,YBTB10} is equivalent to the following proximal ADMM
		\begin{subequations}
			\begin{align}
			q^{k+1} &=  \text{argmax}_{q}L(u^k;p_t^k, p_s^k, q) - \frac{1}{2}\|q-q^k\|_{acI-c\nabla \nabla ^*} \, ,\label{eq:aug:pre1:seg}\\
			p_s^{k+1} &=  \text{argmax}_{p_s}L(u^k;p_t^k, p_s, q^{k+1})\, ,\\
			p_t^{k+1} &=  \text{argmax}_{p_t}L(u^k;p_t, p_s^{k+1}, q^{k+1}) \,,
			\end{align}
		\end{subequations}
		where the weighted norm $\| \cdot \|_{acI-c\nabla \nabla ^*}$ is defined as
		\[
		\|q-q^k\|_{acI - c\nabla \nabla ^*}^2 =  \langle (acI - c\nabla \nabla ^*)(q-q^k) ,q-q^k \rangle.
		\]
	\end{remark}
	Clearly, the parameter $a \geq 8$ should be chosen in order to guarantee non-negativeness of the matrix $acI-c \nabla \nabla ^*$, which is exactly the case  ${1}/{a} = 0.125$ that is employed in \cite{Yuan2010,ybtb2014,YBTB10}.

	

	\subsubsection{ Relaxed Preconditioned Splitting Method of Douglas-Rachford Type}
	
	In this part, we would introduce a relaxed preconditioned splitting method of Douglas-Rachford type~\cite{BS1}, which is particularly designed to efficiently tackle the following primal-dual optimization with a special quadratic term:
	\begin{equation}\label{eq:saddle:qr}
	\min_{x} \max_{y} {F}(x) + \langle {K}x,y \rangle -{G}(y),  
	\end{equation}
	where ${F}(x)  = \langle \frac{1}{2}Qx-f,x \rangle $. 
	
	Each iteration of the relaxed preconditioned splitting method of Douglas-Rachford type for such special type primal-dual optimization problem \eqref{eq:saddle:qr} can be written as:
	\begin{equation}\label{iteration:quadratic} \tag{rPDRQ}
	\begin{cases}
	x^{k+1} = x^{k} + M_{Q}^{-1}[\sigma f  - \sigma {K}^{*} \bar y^{k}- T_{Q} x^{k}], \\
	y^{k+1} = \bar y^k + \tau {K} x^{k+1},  \\
	\bar{y}^{k+1} = \bar y^k + \rho[( I + \tau \partial {G})^{-1}
	(2y^{k+1} - \bar y^k) - y^{k+1}], 
	\end{cases}
	\end{equation}
	where $M_{Q} = N_{1} +\sigma Q + \sigma \tau {K}^* {K}$ is the preconditioner for $T_{Q}= \sigma Q + \sigma \tau {K}^* {K}$ and $\sigma$, $\tau$ are positive sizes that can be chosen freely. The convergence of iterations \eqref{iteration:quadratic} can be guaranteed; see \cite{BS1} for the case $\rho=1$ and \cite{BS} for the case $\rho \in (0,2)$.
	\begin{prop}[\cite{BS1}]\label{pro:pdr}
		Assuming $x \in X$ and $y \in Y$ with $X$, $Y$ being the finite dimensional spaces, if the preconditioner satisfies the feasibility condition, i.e.,  $M_Q \geq T_Q$, then iteration sequence $\{ x^k, y^k\}$ of   the preconditioned Douglas-Rachford splitting \eqref{iteration:quadratic}  converges to a saddle-point $(x^*, y^*)$ of \eqref{eq:saddle:qr}.	
	\end{prop}

	To confirm the convergence of the above algorithm \eqref{iteration:quadratic}, $M_Q \geq T_Q$ is required~\cite{BS}, hence $N_1= M_Q-T_Q$ must be a positive semi-definite matrix. 
	
	Now we consider the equivalent primal-dual formulation \eqref{eq:saddle:qr} with the following data,
	\begin{equation}\label{eq:pdrq:data:f}
	{F}(x) = 0, \  Q = 0, \ x =u, \quad {K} = \begin{pmatrix}-\nabla \\ I \\ -I \end{pmatrix}, \quad y = \begin{pmatrix}q \\p_t\\ p_s\end{pmatrix} \, , 
	\end{equation}
	and 
	\begin{equation}\label{eq:pdrq:data:g}
	{G}(y) = -\langle 1, p_s \rangle +  {I}_{\{p_s \leq C_s\}}(p_s) + {I}_{\{p_t \leq C_t\}}(p_t) + {I}_{\{\|q\|_{\infty} \leq \alpha\}}(q) \, .
	\end{equation}
	Therefore ${K}^* {K}$ is
	\[
	\begin{pmatrix}\D &I& -I\end{pmatrix}\begin{pmatrix}-\nabla \\ I \\ -I \end{pmatrix} = -\Delta + 2I,
	\]
	followed by
	\[
	M_{Q} = N_1 + \sigma \tau {K}^* {K} = N_1 + \sigma\tau(-\Delta + 2I), \quad T_{Q} =  \sigma \tau {K}^* {K} =  \sigma\tau(-\Delta + 2I),
	\]
	with $N_1 \geq 0$. 
	
	Actually, the symmetric red-black Gauss-Seidel (sRBGS) type of algorithm can be used and $N_1 + \sigma \tau  {K}^* {K}$  is just the sRBGS preconditioner for $T_Q$, c.f.~\cite{BS}.
	We thus get the following preconditioned Douglas-Rachford splitting method as \eqref{iteration:quadratic} for \eqref{eq:pd:2:seg}:
	\begin{equation}\label{iteration:quadratic:seg} 
	\begin{cases}
	u^{k+1} = u^{k} + M_{Q}^{-1}[ - \sigma(\D \bar q^k +\bar p^k_t -\bar p^k_s)- T_{Q} u^{k}], \\
	q^{k+1} = \bar q^k - \tau  \nabla u^{k+1},  \\
	p^{k+1}_t = \bar p^k_t + \tau u^{k+1},\\
	p^{k+1}_s = \bar p^k_s  -\tau u^{k+1},\\
	\bar q^{k+1} =  \bar q^k + \rho [\mathcal{P}_{\alpha}(2 q^{k+1} - \bar q^k)  - q^{k+1}], \\
	\bar p^{k+1}_t =  \bar p^k_t + \rho [\mathcal{P}_{C_t}(2 p^{k+1}_t - \bar p^k_t)  - p^{k+1}_t], \\
	\bar p^{k+1}_s =  \bar p^k_s + \rho [\mathcal{P}_{C_s}(2 p^{k+1}_s - \bar p^k_s)  - p^{k+1}_s].
	\end{cases}
	\end{equation}
	One can combine the last six equations to three, eliminating $ q^{k+1},  p^{k+1}_t,  p^{k+1}_s$. For example
	\[
	\bar q^{k+1}   = (1-\rho)\bar q^k +\rho \tau \nabla u^{k+1}+ \rho \mathcal{P}_{\alpha}(-2 \tau \nabla u^{k+1} + \bar q^k) .
	\]
	In numerical tests, we set $\rho = 1.9$.
	The convergence of this relaxed and preconditioned Douglas-Rachford splitting method can be guaranteed. The theory of \cite{BS1} can be used for the convergence analysis.

	\subsection{Novel ADMMs for Multiphase Image Segmentation}\label{sec:multi:seg}
	
	In this section, we mainly focus on the multi-phase case. Although the framework is similar to the two-phase case, the block preconditioners are different due to the more complicated structures. We mainly discuss the over-relaxed ADMM of Eckstein-Bertsekas type \cite{EP} and Fortin-Glowinkis type \cite{FG}. 
	We shall develop several novel and efficient preconditioners.
	
	\subsubsection{Relaxed Augmented Lagrangian Method of Eckstein-Bertsekas Type: Multi-phase case}
	Let's first introduce the augmented Lagrangian functional for \eqref{eq:potts-mf} with constraints \eqref{eq:mpf-cond-01} and \eqref{eq:mpf-cond-03} and the notations $\bf u$, $\bf p$ and $\bf q$ as:
	\bq \label{alm:multi:variables} 
	{\bf u} = (u_1, ..., u_n)^T \, , \quad {\bf q} = (q_1, ..., q_n)^T\, , \quad {\bf \bar p} = (p_1, ..., p_n)^T\, .
	\eq
	The augmented Lagrangian functional can be written as follows:
	\begin{align}
	L({\bf u}; {\bf q}; {\bf p}): =&\langle p_s,1 \rangle -\sum_{i=1}^n I_{\{p_i \leq \rho(l_i,x)\}}(p_i)  \notag -\sum_{i=1}^n I_{\{\|q_i\| \leq \alpha\}}(q_i) \\
	&+ \sum_{i=1}^n \langle u_i, \D q_i +p_i -p_s \rangle   - \frac{c}{2}\sum_{i=1}^n\|\D q_i + p_i -p_s\|_{2}^2. \label{eq:multi:aug:func}
	\end{align}
	We will show how to solve it with the classical two-block ADMM with proximal terms.  The notations $A^*$, $B^*$, $L_n$ and $I_n$ are as follows,
	\begin{equation}\label{eq:AB:notion}
	B^* = \text{Diag}\underbrace{[-\nabla, -\nabla, \cdots, -\nabla]}_{n}, \quad A^* = \begin{bmatrix} 
	I_n \\ -L_n'
	\end{bmatrix}, 
	\end{equation}
	and
	\begin{equation}\label{eq:LI:notion}
	L_n = \underbrace{[I, I, \cdots, I]^{T}}_n, \quad I_n  = \text{Diag}\underbrace{[I, I, \cdots, I]}_n.
	\end{equation}
	Let's further introduce the following block variables and operators:
	\begin{align*}
	& G({\bf p}) = -\langle p_s,1 \rangle +\sum_{i=1}^n I_{\{p_i \leq \rho(l_i,x)\}}(p_i), \quad H({\bf q}) = \sum_{i=1}^n I_{\{\|q_i\| \leq \alpha\}}(q_i),\\
	& A = [I_n, -L_n], \quad  B = \text{Diag}\underbrace{[\D, \D, \cdots, \D]}_n.
	\end{align*} 
	Then the constraint \eqref{eq:mpf-cond-03} can be written as:
	\begin{equation}
	A{\bf p} + B {\bf q} = 0.
	\end{equation}
	The augmented Lagrangian \eqref{eq:multi:aug:func} thus can be reformulated as the following two-block problem:
	\begin{equation}\label{eq:aug:twoblock:multi}
	\min_{{\bf u}}\max_{{\bf q},{\bf p}}L({\bf u}; {\bf q},{\bf p})  := -G({\bf p}) -H({\bf q}) + \langle  {\bf u}, A {\bf p}+ B {\bf q}\rangle - \frac{c}{2}\|A {\bf p}+B {\bf q}\|_{2}^2.
	\end{equation}
	The preconditioning for $\bf q$ is similar to the two-phase case. Let's turn to preconditioning the ${\bf p}$ variable.  By direct calculation, we see
	\[
	A^*A = \begin{bmatrix}
	I_n & -L_n \\
	-L_n' & nI
	\end{bmatrix}.
	\]
	Since $(cA^*A + \partial G)^{-1} $ do not have explicit representations and is hard to invert, specially designed preconditioners are needed. 
	Now, let's introduce our novel diagonal operator for dealing with the implicit equation of the $p$ variable and the corresponding efficient preconditioner.
	\begin{equation}\label{eq:multilabel:pre:p}
	\tilde{A} = \text{Diag}\underbrace{[a_1I, a_1I, \cdots, a_1I, a_2I]}_{n+1},\quad a_1 \geq 2, \ \ a_2 \geq 2n. 
	\end{equation}
	
	The linear operators $N$ and $M$ in \eqref{eq:update:two:sun:lag} are chosen as follows 
	\begin{equation}\label{eq:alm:multi:sun:pre}
	N = acI_n, \quad M = c\tilde A, \quad I_n: = \text{Diag}\underbrace{[I, I, \cdots, I]}_{n},
	\end{equation}
	where $a \geq \|\D^*\D\|$ and $\tilde A$ is the same as in \eqref{eq:multilabel:pre:p}. With preconditioners in \eqref{eq:alm:multi:sun:pre}, denoting $\tilde {\bf q} = (\tilde q_1, \tilde q_2, \cdots, \tilde q_n)^{T}$ and $\tilde {\bf p} = (\tilde p_1, \tilde p_2, \cdots, \tilde p_n ,\tilde p_s)^{T}$, we have 
	\begin{align*}
	(N+\partial H)^{-1}(\tilde q) &=  \left(\mathcal{P}_{\alpha}(\frac{\tilde q_1}{ac}),\mathcal{P}_{\alpha}(\frac{\tilde q_2}{ac}),\cdots,\mathcal{P}_{\alpha}(\frac{\tilde q_n}{ac})\right),\\
	(M+\partial G)^{-1}(\tilde p) &= \left(\mathcal{P}_{\rho(l_1,x)}(\frac{\tilde p_1}{a_1c}),\mathcal{P}_{\rho(l_2,x)}(\frac{\tilde p_2}{a_1c}),,\cdots,\mathcal{P}_{\rho(l_n,x)}(\frac{\tilde p_n}{a_1c}), \frac{\tilde p_s +1}{a_2c}\right).
	\end{align*}
	
	For the choice of $M$ and $N$ for the updating of $\bf p$ and $\bf q$ of the relaxed and preconditioned ADMM \eqref{eq:update:two:sun:lag}, we have the following theorem. 
	\begin{theorem}\label{lem:digonal:p:multi:4}
		For the diagonal operators $\tilde A$, we have  $ \tilde{A} \geq A^*A$.
		We thus choose $M=acI_n$ and $N=c\tilde A$ with $a=8$ satisfying the condition in \eqref{eq:con:sun:pradmm}. Assuming $\{\rho_k \in (0,2)\}$ is a non-decreasing sequence, then $(\bf q^k,\bf p^k, \bf u^k)$ converges weakly
		to a saddle-point $(\bf q^*, \bf p^*, \bf u^*)$  of \eqref{eq:aug:twoblock:multi} and $(\bf q^*, \bf p^*)$ is a solution of \eqref{eq:potts-mf}.
	\end{theorem}
	\begin{proof}
		We first show that for any ${\bf p}=(p_1,p_2, \cdots, p_n, p_s)^{T}$, 
		\[
		\langle \tilde A {\bf p}, {\bf p} \rangle \geq \langle {\bf p}, A^*A {\bf p} \rangle.  
		\]
		With direct calculation, we have
		\begin{align*}
		\langle {\bf p}, A^*A {\bf p} \rangle &= \langle (p_1,p_2, \cdots, p_n, p_s)^{T}, (p_1-p_s, p_2 -p_s, \cdots, p_n-p_s, -\sum_{i=1}^np_i  + np_s)^{T} 
		\rangle \\
		& =\sum_{i=1}^n p_i(p_i-p_s) + p_s(-\sum_{i=1}^np_i  + np_s) \\
		& = \sum_{i=1}^n(p_i^2 -2p_ip_s) + np_s^2   \leq 2\sum_{i=1}^n p_i^2 + 2np_s^2 \\
		&\leq \langle \tilde A {\bf p},  {\bf p}\rangle. 
		\end{align*}
		The convergence follows similarly to Theorem \ref{thm:twolabel}.
	\end{proof}

	With these preparations, for $n=4$, writing the algorithm \ref{eq:update:two:sun:lag} for the 4-labeling case \eqref{eq:potts-mf}  component-wisely, we have
	\begin{subequations}
		\begin{align}
		q_i^{k+1} &= \mathcal{P}_{\alpha}\left\{ (I-\frac{1}{a}{\D}^{*}\D)q_i^k  + \frac{1}{a} \nabla(p_i^k -{p_s}^k -\frac{u_i^k}{c})\right\}, \  \ i=1,2,3,4, \\
		p_i^{k+1}&= \mathcal{P}_{\rho(l_i,x)}\left(p_{i}^k  - \frac{1}{a_1}(\rho_k(p_i^k-p_s^k) -\frac{1}{c}u_i^k+\rho_k\D q_i^{k+1} )\right), \ i=1,2,3,4, \notag \\
		p_s^{k+1} & = p_s^{k} + \frac{1}{a_2} \left( \sum_{i=1}^4 \rho_k\D q_i^{k+1} + \sum_{i=1}^4(\rho_k(p_i^k-{p_s}^k )-\frac{1}{c}u_i^k) + \frac{1}{c} \right), \\
		u_i^{k+1} & = {u_i}^k - c\left(\rho_k\D q_i^{k+1} + \rho_k(p_i^{k+1} - p_s^{k+1})-(1-\rho_k)(p_i^k-p_s^k)\right), \ \ i=1,2,3,4. \notag 
		\end{align}
	\end{subequations}
	Here we choose $a=8$, $a_1 = 2$ and $a_2=8$ according to \eqref{eq:multilabel:pre:p} and Theorem \ref{lem:digonal:p:multi:4} and the over-relaxation parameter $\rho_k \equiv 1.9$.
	
	\subsubsection{Relaxed Preconditioned ADMM of Fortin-Glowinski Type: Multi-phase case}
	
	Similar to the two-phase case, we start from the augmented Lagrangian \eqref{eq:aug:twoblock:multi}. For the $\bf q$ block, we will deal with it as the two-block case, i.e.,
	\[
	{\bf q}^{k+1} =  \text{argmax}_{{\bf q}} L( {\bf u}; {\bf q},{\bf p})  - \frac{1}{2} \|{\bf q}-{\bf q}^k\|_{acI_n - cB^*B},
	\] 
	where $a \geq \|\D^* \D\|$. 
	For the $p$ variable, we can introduce the proximal term
	$\frac{c}{2}\|{\bf p}-{\bf p}^k\|_{c\tilde A - cA^*A}$. The proximal terms satisfy the conditions
	\[
	aI_n \geq  B^*B, \quad \tilde A - A^*A,
	\]
	with $a \geq 8$ and $\tilde A$ as in \eqref{eq:multilabel:pre:p} and the convergence follows.  We finally obtain the proximal ADMM for \eqref{eq:potts-mf} as follows,
	\begin{equation}\label{eq:aug:multi:prox:seg}
	\begin{cases}
	{\bf q}^{k+1} =  \text{argmax}_{{\bf q}}L( {\bf u}^k;{\bf p}^k, {\bf q}) - \frac{1}{2}\|{\bf q}-{\bf q}^k\|_{acI_n-cB^*B},\\
	{\bf p}^{k+1} =  \text{argmax}_{{\bf p}}L( {\bf u}^k; {\bf p}, {\bf q}^{k+1}) -\frac{1}{2}\|{\bf p}- {\bf p}^k\|_{c\tilde A - cA^*A},\\
	{\bf u}^{k+1} =  {\bf u}^k - rc(A{\bf p}^{k+1} + B{\bf q}^{k+1}),
	\end{cases}
	\end{equation}
	where $r$ is the relaxation parameter with $r \in (0, \frac{\sqrt{5}+1}{2})$. For the 4 labels ($n=4$) case, we choose
	\[
	\tilde A = \text{Diag}[a_1I, a_1I, a_1I, a_1I, a_2I], \quad a_1=2, \quad a_2 = 8. 
	\] 
	Writing \eqref{eq:aug:multi:prox:seg} component-wisely, we arrive at 
	\begin{equation}\label{eq:multi:fg:4}
	\begin{cases}
	q_i^{k+1} = \mathcal{P}_{\alpha}\{ (I-\frac{1}{a}{\D}^{*}\D)q_i^k  + \frac{1}{a} \nabla(p_i^k -{p_s}^k -\frac{u_i^k}{c})\}, \  \ i=1,2,3,4, \\
	p_i^{k+1}= \mathcal{P}_{\rho(l_i,x)}\left(p_{i}^k  - \frac{1}{a_1}(p_i^k-p_s^k -\frac{1}{c}u_i^k+\D q_i^{k+1} )\right), \ i=1,2,3,4, \\
	p_s^{k+1}  = p_s^{k} + \frac{1}{a_2} \left( \sum_{i=1}^4 \D q_i^{k+1} + \sum_{i=1}^4(p_i^k-{p_s}^k -\frac{1}{c}u_i^k) + \frac{1}{c} \right), \\
	u_i^{k+1}  = {u_i}^k - rc(\D q_i^{k+1} + p_i^{k+1} - p_s^{k+1}), \ \ i=1,2,3,4.
	\end{cases}
	\end{equation}
	
	Similar to Remark \ref{remark:two},  with the augmented Lagrangian functional \eqref{eq:multi:aug:func}, the classical ALM framework introduced in \cite{BYT,ybtb2014} for the convex relaxed Potts model \eqref{eq:potts-mf} is equivalent to the following proximal multi-block ADMM method:
	\begin{equation}\label{eq:mADMM:Tai_yuan} 
	\begin{cases}
	q_i^{k+1} = \mathcal{P}_{\alpha}\left( (I-\frac{1}{a}\nabla \nabla ^*)q_i^k  + \frac{1}{a} \nabla(p^k_i -p^k_s -\frac{u_i^k}{c})\right), \ \ i=1, \cdots, n,\\
	p_{i}^{k+1} = \mathcal{P}_{\rho(l_i,x)}(p_{s}^{k} - \D q_i^{k+1}+ \frac{u_i^k}{c}), \ \ i =1, \cdots, n, \\
	p_{s}^{k+1} = \sum_{i=1}^n(p_{i}^k + \D q_i^{k+1} - u_i^k/c)/n+ \frac{1}{nc},  \\
	u_i^k = u_i^k - c(p_i^{k+1} - p_s^{k+1} + \D q_i^{k+1}) \, , \ \ i=1, \cdots, n,
	\end{cases}
	\end{equation}
	This gives the same explanation of the parameter $a$ as in Remark \ref{remark:two}. Clearly, there is no convergence guarantee due to such a sequential multi-block optimization structure \cite{CBHY}! Compared to the iteration \eqref{eq:mADMM:Tai_yuan},  it can be seen that the iteration \eqref{eq:multi:fg:4} is very compact and there is nearly no extra computational effort with convergence guarantee.

	\subsubsection{Relaxed Preconditioned Splitting Method of Douglas-Rachford Type: Multi-phase case}
	For the preconditioned Douglas-Rachford splitting method, we need the saddle-point formulation \eqref{eq:saddle:qr} with the following data for  \eqref{eq:potts-model},
	\bq \label{eq:g:pd:form}
	G(y) = -\langle p_s,1 \rangle +\sum_{i=1}^n I_{\{p_i \leq \rho(l_i,x)\}}(p_i)  +\sum_{i=1}^n I_{\{\|q_i\| \leq \alpha\}}(q_i),
	\eq
	together with
	\begin{align}\label{eq:saddle:multi:oper:notion}
	& F(x) =0, \quad x = {\bf u}, \quad y = ({\bf q}, {\bf p} )^{T}, \quad 
	K = \begin{bmatrix}
	B^* \\A^*
	\end{bmatrix}, \ \ K^* = [B,A],
	\end{align}
	where the notations ${\bf u}$, ${\bf q}$, ${\bf p}$ are the same as in \eqref{alm:multi:variables}.
	With the same notations in \eqref{eq:saddle:multi:oper:notion}, \eqref{eq:AB:notion}, \eqref{eq:LI:notion} and \eqref{eq:g:pd:form}, let's first calculate $K^*K$. It can be verified that
	\begin{align}
	K^*K &= [B,A]\begin{bmatrix}
	B^* \\A^*
	\end{bmatrix} = BB^* + AA^* = \text{Diag} \underbrace{[-\Delta, -\Delta, \cdots, -\Delta]}_{n}  \\
	& = \text{Diag} \underbrace{[-\Delta, -\Delta, \cdots, -\Delta]}_{n}  + \text{Diag} \underbrace{[I, I, \cdots, I]}_{n} + \text{Ones(n,n)},
	\end{align}
	where $\text{Ones(n,n)}$ is a $n\times n$ operator matrix with each element being $I$. Solving the linear equation involving with $K^*K$ is very challenging. Efficient preconditioners are of the critical importance. Fortunately, we have the following lemma, which can bring out an efficient preconditioner.
	\begin{Lem}\label{lem:dr:pre:multi}
		We can choose the following $T_0$ as a feasible preconditioner for $K^*K$,
		\[
		T_0  \geq   K^*K, 
		\]	
		where $T_0 = \emph{Diag} \underbrace{[-\Delta + (n+1)I, -\Delta + (n+1)I, \cdots, -\Delta + (n+1)I]}_n$.
	\end{Lem}
	\begin{proof}
		Actually, we just need to prove that for any ${\bf u}$, 
		\[
		\langle T_0 {\bf u},{\bf u} \rangle \geq \langle {\bf u}, K^*K{\bf u} \rangle.  
		\]
		By direct calculation, we obtain
		\begin{align*}
		\langle {\bf u}, K^*K{\bf u} \rangle & = \langle (u_1, u_2, \cdots, u_n)^{T}, (\sum_{i=1}^nu_i, \sum_{i=1}^nu_i, \cdots,   \sum_{i=1}^n u_i)^{T}\rangle \\
		&+   \langle (u_1, u_2, \cdots, u_n)^{T}, (-\Delta u_1 +u_1,   -\Delta u_2 +u_2, \cdots, -\Delta u_n +u_n )^{T}\rangle \\
		&=\sum_{i=1}^n \langle -\Delta u_i, u_i \rangle + \sum_{i=1}^n \langle u_i,u_i \rangle  + \sum_{j=1}^n(\sum_{i=1}^n u_i)u_j \\
		& = \sum_{i=1}^n \langle (-\Delta +2I) u_i, u_i \rangle + \sum_{j=1}^n \sum_{i=1,i \neq j}^n u_i u_j \\
		& = \sum_{i=1}^n \langle (-\Delta +2I) u_i, u_i \rangle + 2\sum_{j=1}^n \sum_{i=1,i <j}^n u_i u_j  \\
		& \leq  \sum_{i=1}^n \langle (-\Delta +2I) u_i, u_i \rangle + \sum_{j=1}^n \sum_{i=1,i <j}^n (u_i^2 + u_j^2) \\ 
		&= \sum_{i=1}^n \langle (-\Delta +(n+1)I) u_i, u_i \rangle =  \langle T_0 {\bf u}, {\bf u} \rangle.
		\end{align*}
	\end{proof}

	Lemma \ref{lem:dr:pre:multi} can help us to design an efficient preconditioner for $T_0$ instead of $K^*K$. For the preconditioned and relaxed Douglas-Rachford splitting method \eqref{iteration:quadratic}, since $Q=0$ and $f=0$, we have $T_Q = \sigma \tau K^*K$. Supposing $M$ is the symmetric red-black Gauss-Seidel preconditioner for $\sigma \tau T_0$, we have \cite{BS}
	\begin{equation}
	M \geq  \sigma \tau T_0.
	\end{equation}
	Since $\sigma \tau T_0 \geq \sigma \tau K^*K$, we then obtain
	\begin{equation}
	M \geq  \sigma \tau K^*K,
	\end{equation}
	i.e., $M$ is also a feasible preconditioner for $\sigma \tau K^*K$ \cite{BS}. However since $\sigma \tau K^*K$ is not a diagonal operator, designing efficient preconditioners for $\sigma \tau K^*K$ directly is a subtle issue. Fortunately, by \cite{BS}, we can perform the preconditioned iteration in \eqref{iteration:quadratic} as follows,
	\begin{subequations}\label{eq:preconditioned:alternative}
		\begin{align}
		{\bf u}^{k+1} &= {\bf u}^k + M^{-1}[-\sigma K^*\bar y^k - \sigma \tau K^*K {\bf u}^k] \\
		&=  {\bf u}^k + M^{-1}[(-\sigma K^*\bar y^k -\sigma \tau (K^*K - T_0 ){\bf u}^k)- \sigma \tau T_0 {\bf u}^k].
		\end{align}
	\end{subequations}
	Furthermore, we denote 
	\begin{equation}
	b^k: = -\sigma K^*y^k -\sigma \tau (K^*K - T_0 ){\bf u}^k.
	\end{equation}
	Then \eqref{eq:preconditioned:alternative} finally becomes the classical preconditioned iteration
	\[
	{\bf u}^{k+1} = {\bf u}^k + M^{-1}[b^k -  \sigma \tau T_0 {\bf u}^k],
	\]
	where $M$ is a preconditioner for the diagonal operator $T_0$ and \eqref{eq:preconditioned:alternative} is one step preconditioned iteration  in computations for dealing with the following modified equation
	\[
	T_0 {\bf u}^{k+1} = b^k.
	\]
	Finally, for the 4 labels case with $n=4$, with these preparations and the use of  \eqref{iteration:quadratic}, we have the following preconditioned and relaxed Douglas-Rachford iterative algorithm: 
	\begin{equation}\label{eq:pdrq:4labels}
	\begin{cases}
	b_i^k  = -\sigma (\D \bar q_i^k + \bar p_i^k - \bar p_s^k) + \sigma \tau (3u_i^k -\sum_{j=1,j\neq i}^4 u_j^k), \quad  i =1, 2, 3,4, \\
	u_i^{k+1}  = u_i^k + M_{u}^{-1}[b_i^k - (5\sigma \tau I -\sigma \tau \Delta)u_i^k], \quad i=1,2,3,4, \\
	q_i^{k+1} = \bar q_i^k - \tau \nabla u_{i}^{k+1}, \quad i=1,2,3,4 \\
	p_i^{k+1} = \bar p_i^k + \tau u_i^{k+1}, \quad i=1,2,3,4, \\
	p_s^{k+1}  = \bar  p_s^k -\tau \sum_{i=1}^4u_i^{k+1}, \quad \\
	\bar q_i^{k+1}   = \bar q_i^k + \rho[\mathcal{P}_{\alpha}(2q_i^{k+1} - \bar q_i^k) - q_i^{k+1} ], \quad i=1,2,3,4, \\
	\bar p_i^{k+1}  = \bar p_i^k + \rho[\mathcal{P}_{\rho(l_i,x)}(2p_i^{k+1} - \bar p_i^k) - p_i^{k+1}], \quad i=1,2,3,4, \\
	\bar p_s^{k+1}  = \bar p_s^k + \rho[(2 p_s^{k+1}- \bar p_s^k +\tau) - p_s^{k+1}].
	\end{cases}
	\end{equation}
	
	%
	%
	%
	
	
	\section{Numerical experiments}\label{sec:num:experi}
	
	For all the experiments, we choose $\alpha = 0.5$ for the total variation regularization. The first-order primal-dual algorithm of \cite{CP} is chosen for comparison. The detail of the primal-dual algorithm for the convex relaxed min-cut model \eqref{seg:chanvese} and the convex relaxed Potts model \eqref{eq:relaxed-potts} is given in detail in the Appendix \ref{sec:appendix}. The parameter settings of the algorithms used in our experiments are as follows:
	\begin{itemize}
		\item For ALG1,  
		the primal-dual algorithm
		introduced in \cite{CP} with constant step sizes: $\sigma=0.4$, $\tau = 1/(L^2\sigma)$ with $L = \sqrt{10}$ in \eqref{iteration:pd:seg} for the two-labels case;  $\sigma=0.4$, $\tau = 1/(L^2\sigma)$ with $L = \sqrt{13}$ in \eqref{eq:alg1:4labels}for the four-labels case.
		\item For pADMM-TY, pADMMI, rpADMMI and rpADMMII: we choose $c=0.3$ for both the two labels and four labels cases. Here pADMMI is the preconditioned and relaxed ADMM of Fortin-Glowinksi type in \eqref{eq:update:two:tai:lag} without relaxation, i.e.,  the relaxation parameter $r=1.0$. pADMM-TY denotes the ADMM as in \eqref{eq:ADMM:Tai_yuan} for the two labels case or  \eqref{eq:mADMM:Tai_yuan} for the multi-labeling case.
		\item For rPDRQ: we choose $\sigma = 0.2$, $\tau = 1.0$ in \eqref{iteration:quadratic:seg} for the two-labels case; $\sigma=5$, $\tau=0.4$ for the four-labels case in \eqref{eq:pdrq:4labels}.
	\end{itemize}
	\begin{table}
		\centering 
		\begin{tabular}{lr@{\,}r@{\,}lr@{\,}r@{\,}lr@{\,}r@{\,}lr@{\,}r@{\,}l} 
			\toprule
			& \multicolumn{6}{c}{shooter}
			& \multicolumn{6}{c}{yuanbo}\\
			\cmidrule{2-7} 
			\cmidrule{8-13}
			& \multicolumn{3}{c}{$\varepsilon = 10^{-4}$}
			& \multicolumn{3}{c}{$\varepsilon = 10^{-6}$}
			& \multicolumn{3}{c}{$\varepsilon = 10^{-4}$}
			& \multicolumn{3}{c}{$\varepsilon = 10^{-6}$}\\
			\cmidrule{1-13} 
			pADMM-TY && 164&(1.70s) && 3633&(34.97s) &&  240&(80.04)
			&&1480&(498.70s) \\
			ALG1&& 190&(1.93s)  && 4435&(43.40s) && 283&(89.27s) && 1961&(648.33s)\\
			\midrule 
			rpADMMI &&129&(1.54s) &&2403&(24.89s) &&188&(62.61s) && 1111&(377.00s)\\
			rpADMMII &&121&(1.40s) &&2084&(21.94s) &&176&(61.43s) &&1059&(377.28s)\\
			rPDRQ && 94&(1.11s)  && 1722&(18.56s) && 141&(49.34s) && 952&(322.66s)\\
			\bottomrule 
		\end{tabular}
		
		\vspace*{-0.5em}
		\caption{ Numerical results for the TV-regularized image segmentation with regularization parameter
			$\alpha = 0.5$.
			The iteration is performed until the relative error of primal energy
			is below $\varepsilon$. Two labels case.}
		\label{tab:2:labels}
	\end{table}

	\begin{table}
		\centering 
		\begin{tabular}{lr@{\,}r@{\,}lr@{\,}r@{\,}lr@{\,}r@{\,}lr@{\,}r@{\,}l} 
			\toprule
			& \multicolumn{6}{c}{brain}
			& \multicolumn{6}{c}{butterfly}\\
			\cmidrule{2-7} 
			\cmidrule{8-13}
			& \multicolumn{3}{c}{$\varepsilon = 10^{-4}$}
			& \multicolumn{3}{c}{$\varepsilon = 10^{-5}$}
			& \multicolumn{3}{c}{$\varepsilon = 10^{-4}$}
			& \multicolumn{3}{c}{$\varepsilon = 10^{-5}$}\\
			\cmidrule{1-13} 
			pADMMI && 504&(39.57s) && 1679&(132.90s) &&  429&(31.61)
			&&1428&(105.45s) \\
			ALG1&& 698&(56.93s)  && 2244&(181.68s) && 644&(48.60s) && 1997&(155.38s)\\
			\midrule 
			rpADMMI &&361&(28.36s) &&1282&(104.22s) &&309&(22.88s) && 1148&(85.02s)\\
			rpADMMII &&323&(29.14s) &&1311&(116.53s) &&280&(24.36s) && 1083&(90.51s)\\
			rPDRQ && 318&(29.89s)  && 1089&(104.84s) && 282&(25.22s) && 948&(83.81s)\\
			\bottomrule 
		\end{tabular}
		
		\vspace*{-0.5em}
		\caption{ Numerical results for the TV-regularized image segmentation model with regularization parameter
			$\alpha = 0.5$.
			The iteration is stopped  when the relative error of primal energy
			is below $\varepsilon$ for this Multi-phase case.}
		\label{tab:4:labels}
	\end{table}

	\begin{figure}
		
		\begin{center}
			\subfloat[Original image]
			{\includegraphics[width=0.31\textwidth]{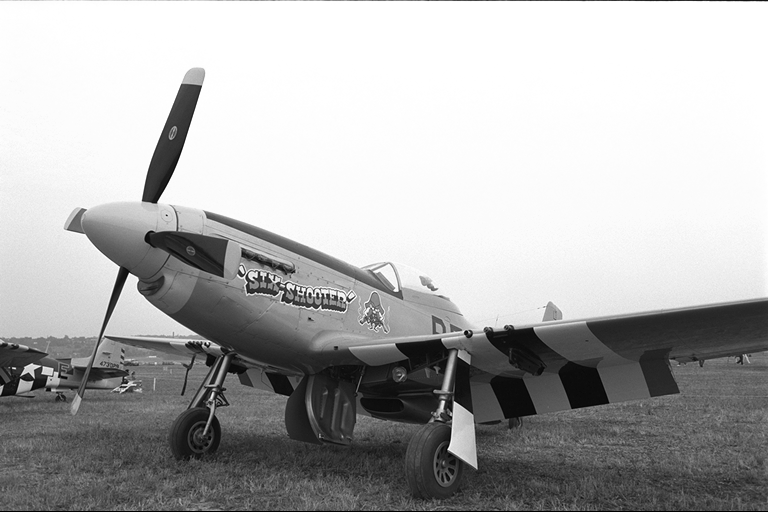}}\ \ \ 
			\subfloat[rpADMMI,  $\varepsilon = 10^{-4}$]
			{\includegraphics[width=0.31\textwidth]{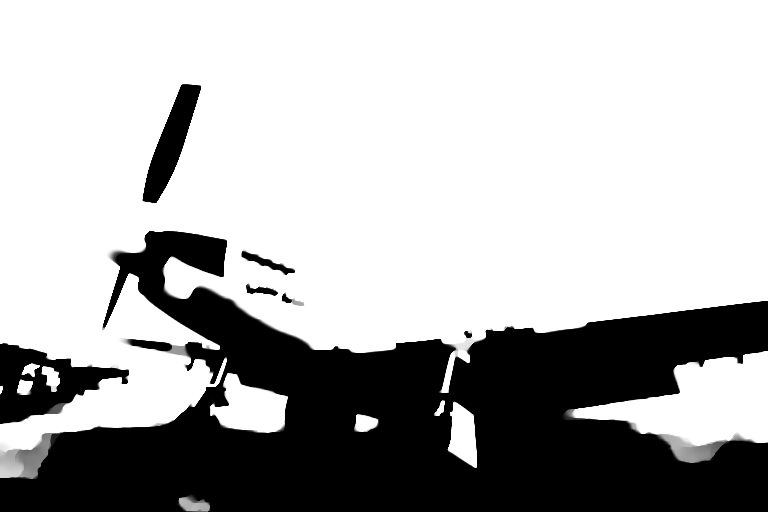}}\ \  \
			\subfloat[rpADMMI, $\varepsilon = 10^{-6}$]
			{\includegraphics[width=0.31\textwidth]{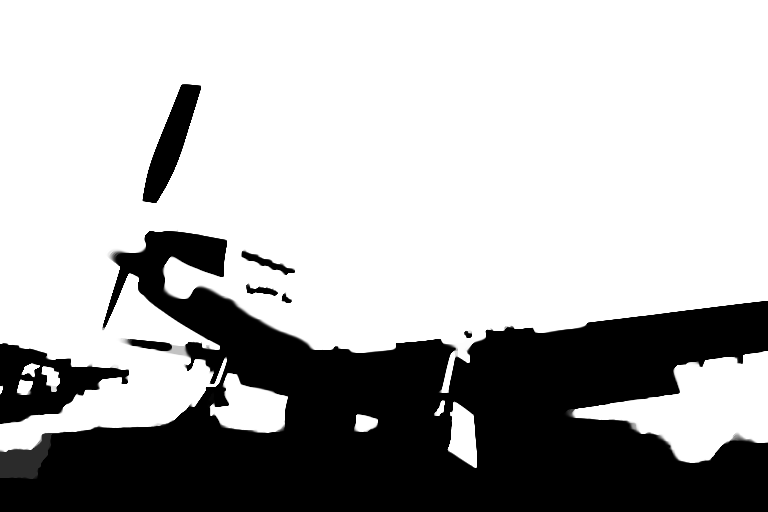}}\\ [-0.1em]
			\subfloat[Original image]
			{\includegraphics[width=0.31\textwidth]{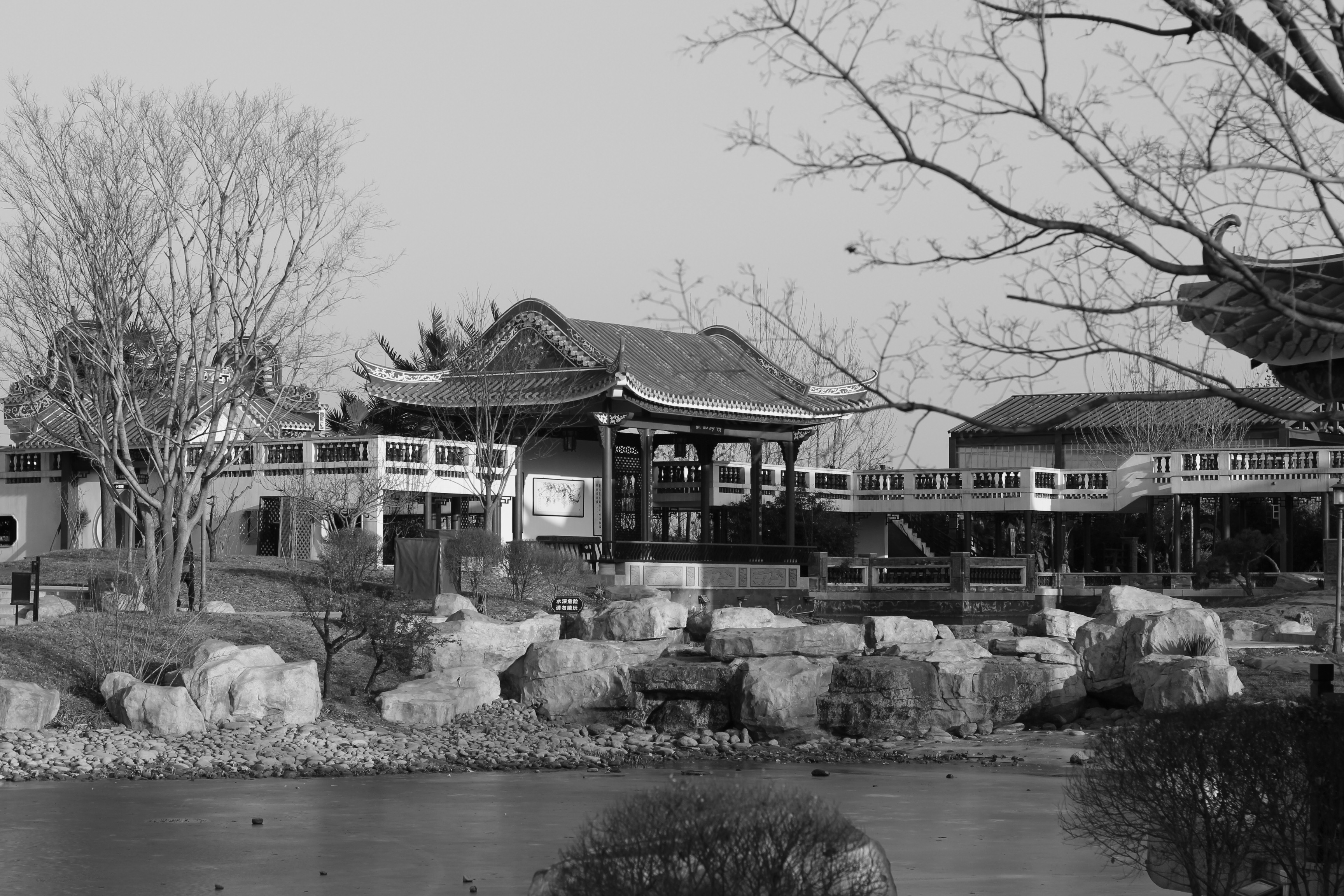}}\ \ \ 
			\subfloat[rpADMMII,  $\varepsilon = 10^{-4}$]
			{\includegraphics[width=0.31\textwidth]{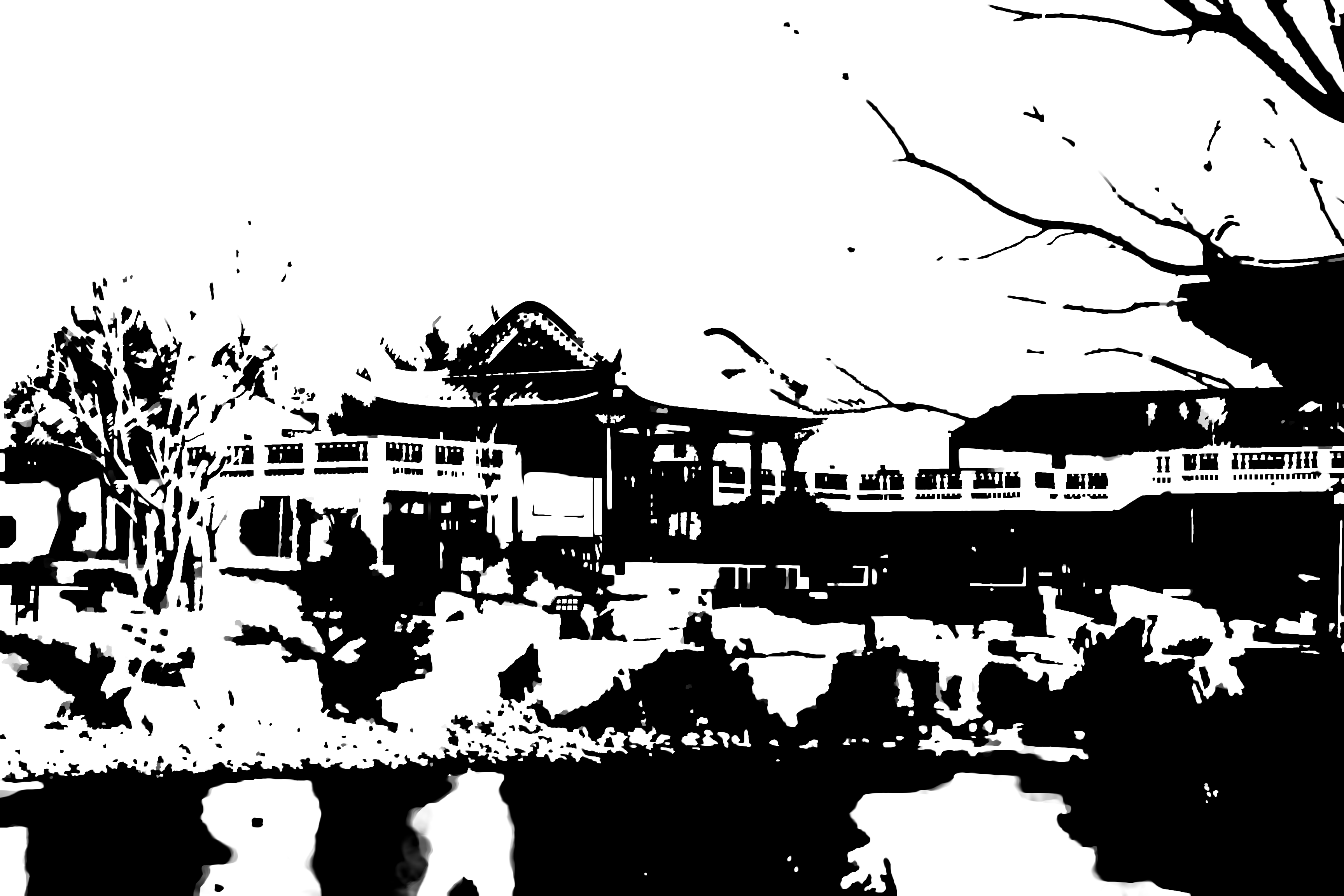}}\ \ \ 
			\subfloat[rpADMMII, $\varepsilon = 10^{-6}$]
			{\includegraphics[width=0.31\textwidth]{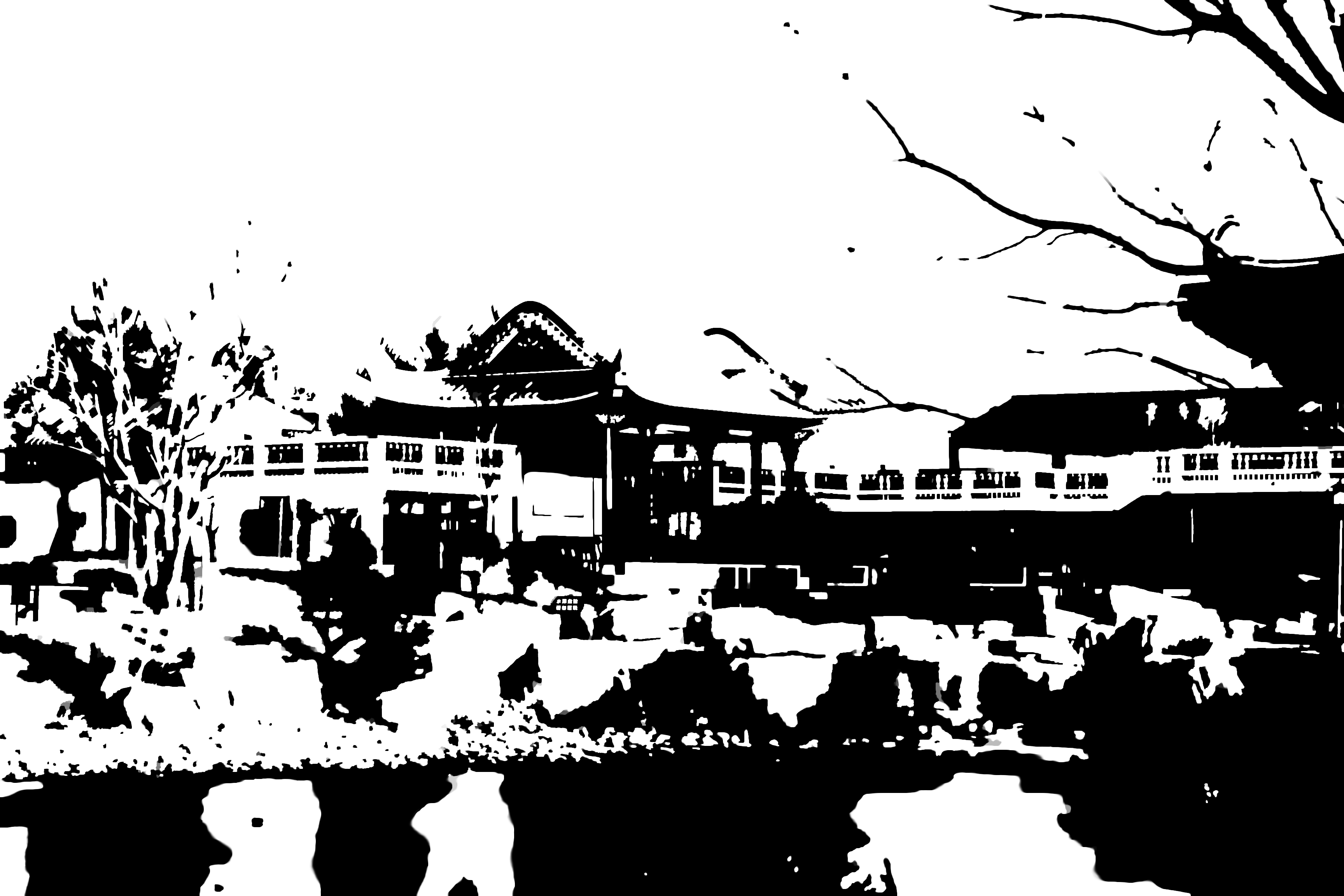}}
		\end{center}
		\vspace*{-0.5em}
		\caption{Results for TV-regularized image segmentation with $\alpha = 0.5$ by rPDRQ. (c) and (f) are the denoised images with $\alpha = 0.5$ 
			respectively.}
		\label{brain:seg:two}
	\end{figure}

	\begin{figure}
		\begin{center}
			\subfloat[Numerical convergence rate of relative primal energy compared with iteration number.]
			{\includegraphics[width=5.8cm]{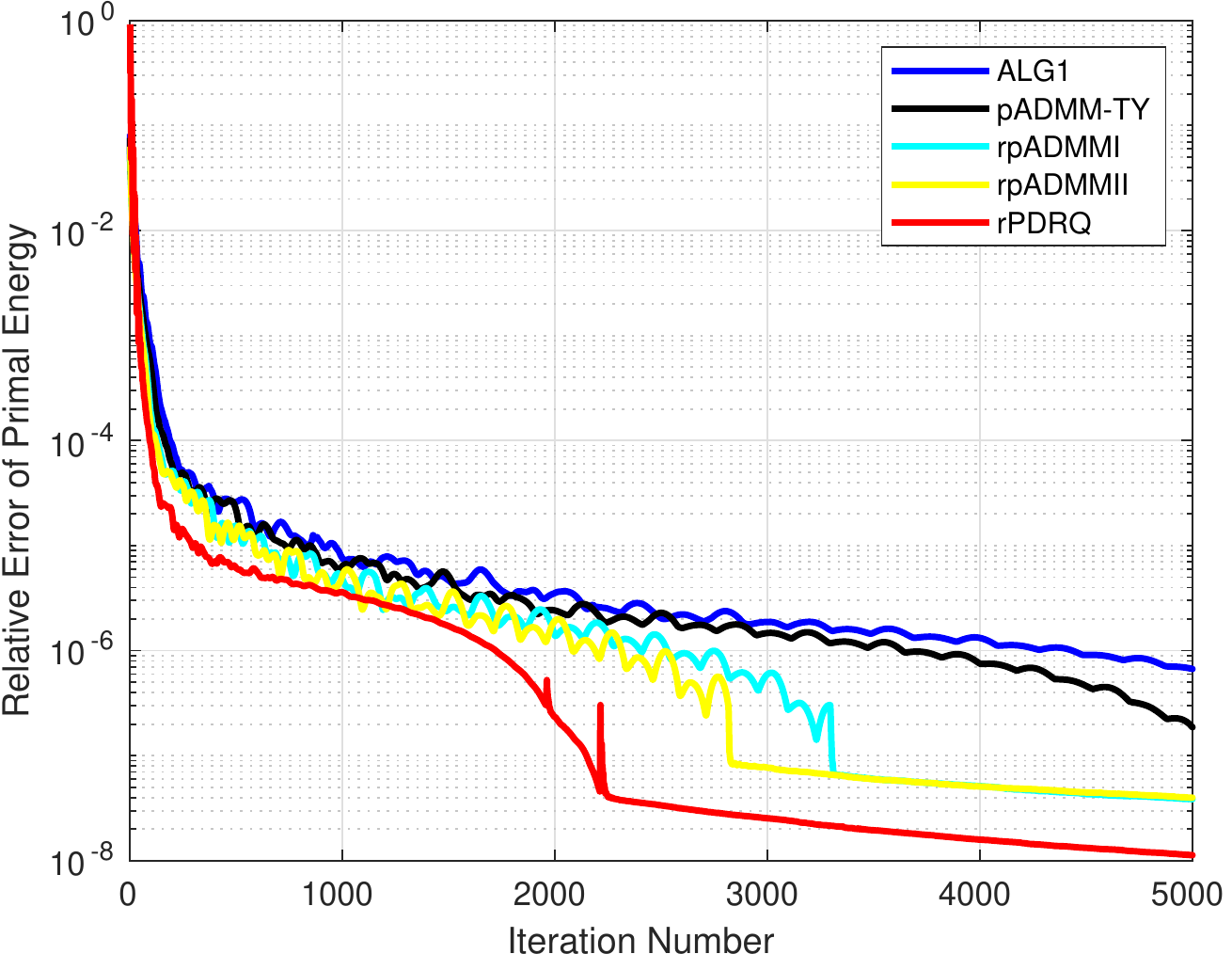}}\hfill
			\subfloat[Numerical convergence rate of relative primal energy 
			compared with iteration time.]
			{\includegraphics[width=5.6cm]{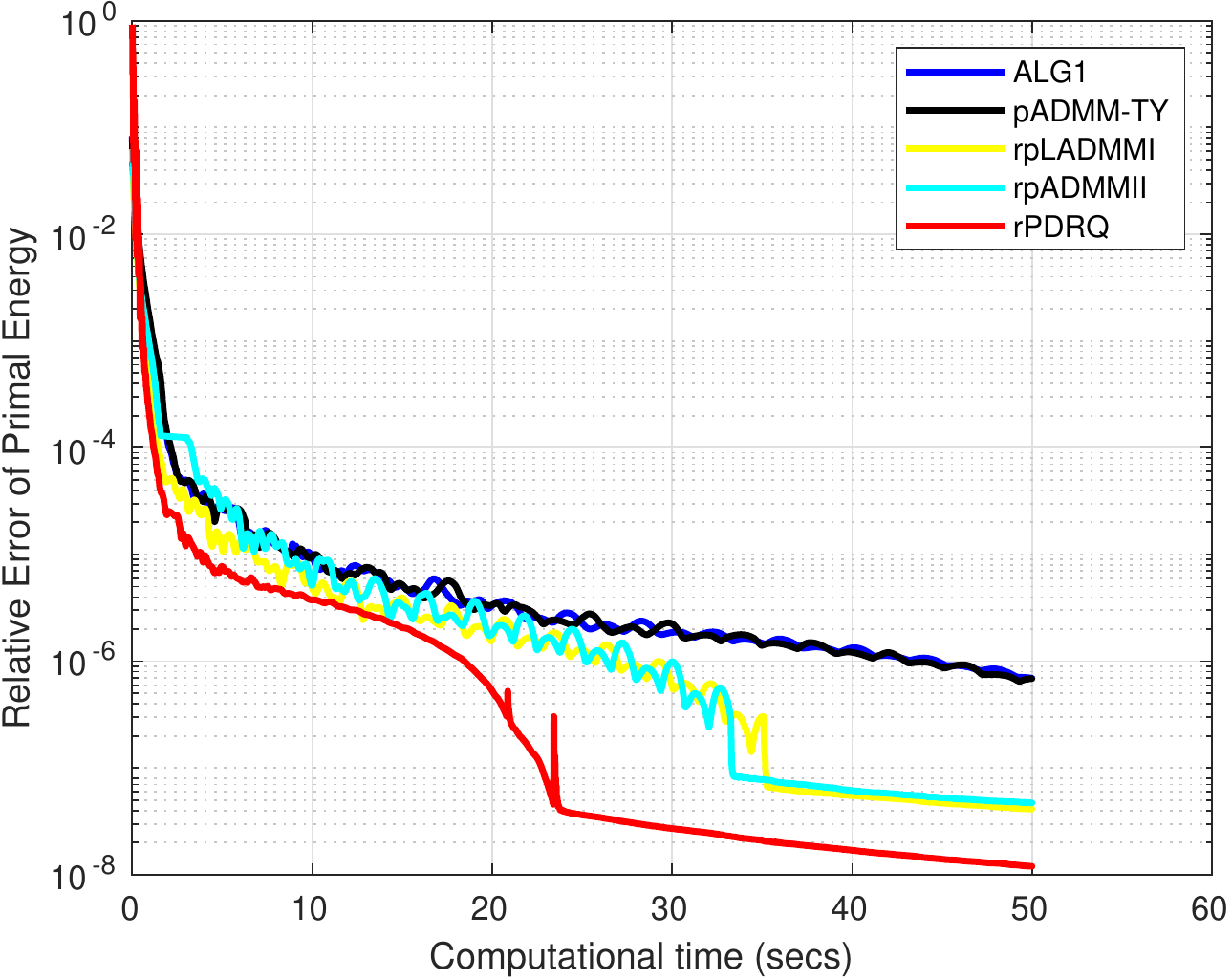}}
		\end{center}
		\caption{$\TV$-regularized image segmentation: numerical convergence rates.
			The relative error of primal energy is compared in terms of iteration
			number and computation time for
			Figure~\ref{tulips:denoise}(d)
			with  $\alpha = 0.5$. Note the 
			semi-logarithmic scale is used in the plot respectively.}
		\label{fig:shooter:performance}
	\end{figure}

	\begin{figure}
		
		\begin{center}
			\subfloat[Original image]
			{\includegraphics[width=0.31\textwidth]{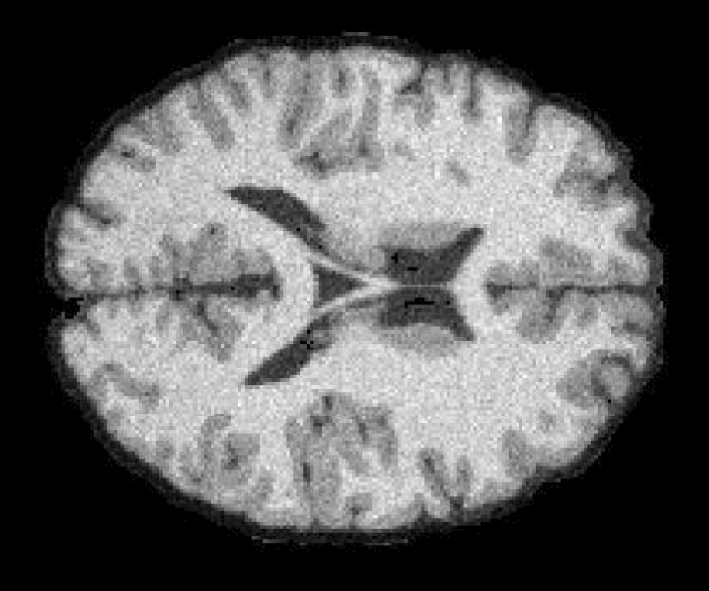}}\ \ \ 
			\subfloat[rpADMMI,  $\varepsilon = 10^{-4}$]
			{\includegraphics[width=0.31\textwidth]{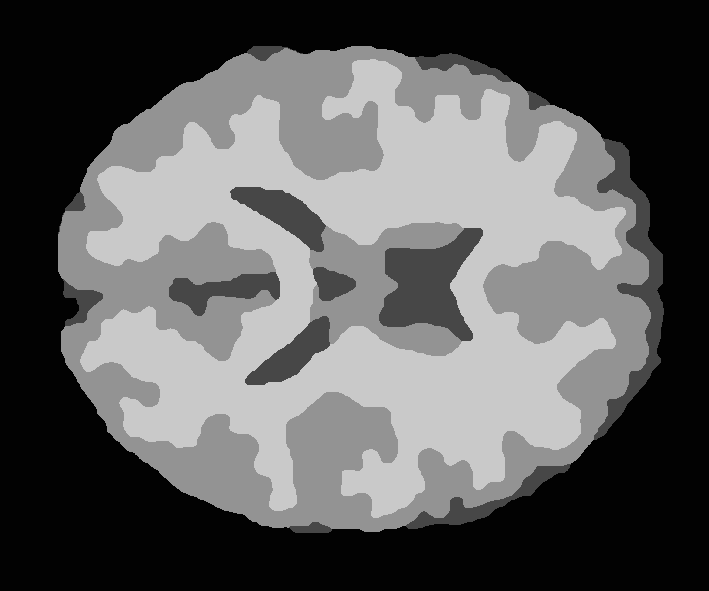}}\ \  \
			\subfloat[rpADMMI, $\varepsilon = 10^{-6}$]
			{\includegraphics[width=0.31\textwidth]{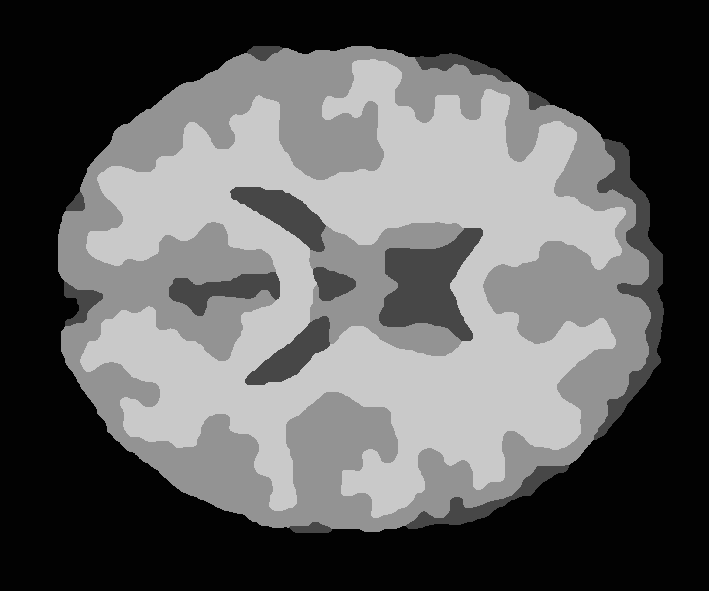}}\\ [-0.1em]
			\subfloat[Original image]
			{\includegraphics[width=0.31\textwidth]{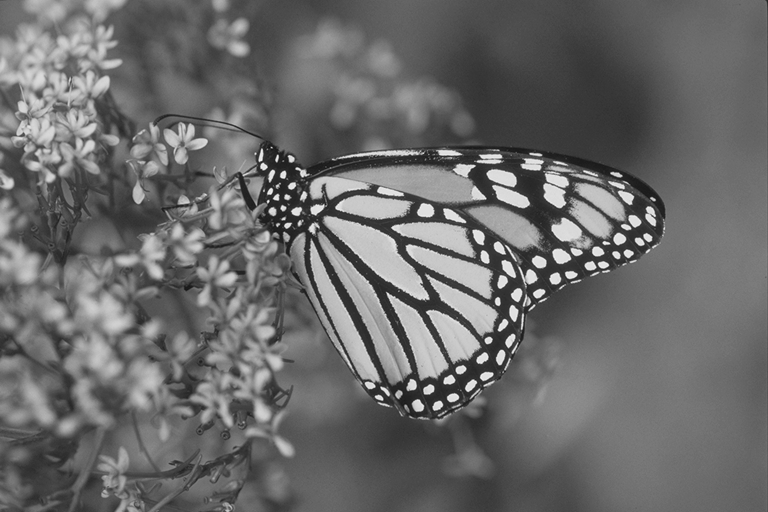}}\ \ \ 
			\subfloat[rpADMMII,  $\varepsilon = 10^{-4}$]
			{\includegraphics[width=0.31\textwidth]{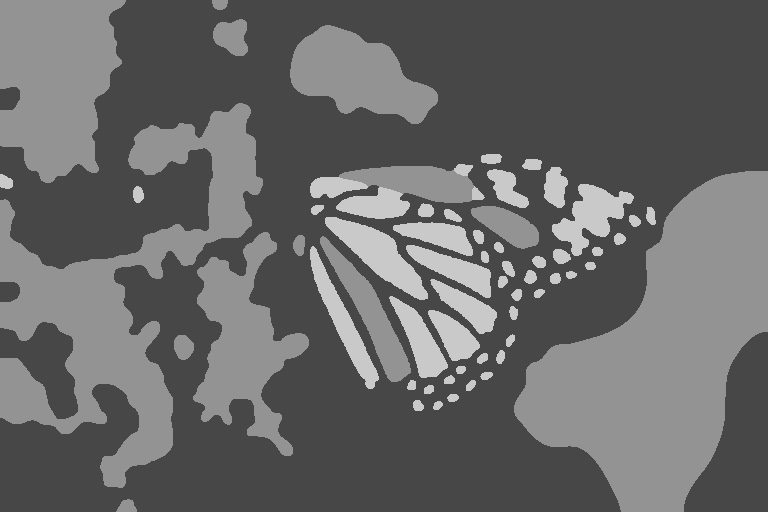}}\ \ \ 
			\subfloat[rpADMMII, $\varepsilon = 10^{-6}$]
			{\includegraphics[width=0.31\textwidth]{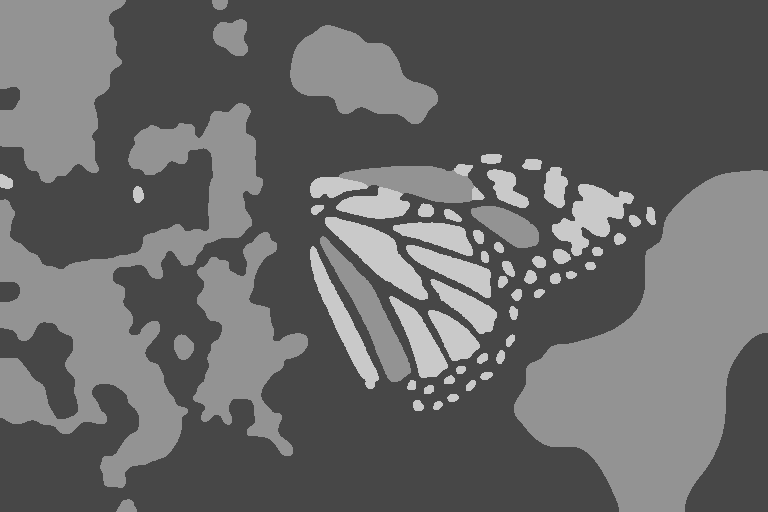}}
		\end{center}
		\vspace*{-0.5em}
		\caption{Results for TV-regularized image segmentation with $\alpha = 0.5$ by rPDRQ. (c) and (f) are the denoised images with $\alpha = 0.5$ 
			respectively.}
		\label{tulips:denoise}
	\end{figure}

	\begin{figure}
		\begin{center}
			\subfloat[Numerical convergence rate of relative primal energy compared with iteration number.]
			{\includegraphics[width=5.8cm]{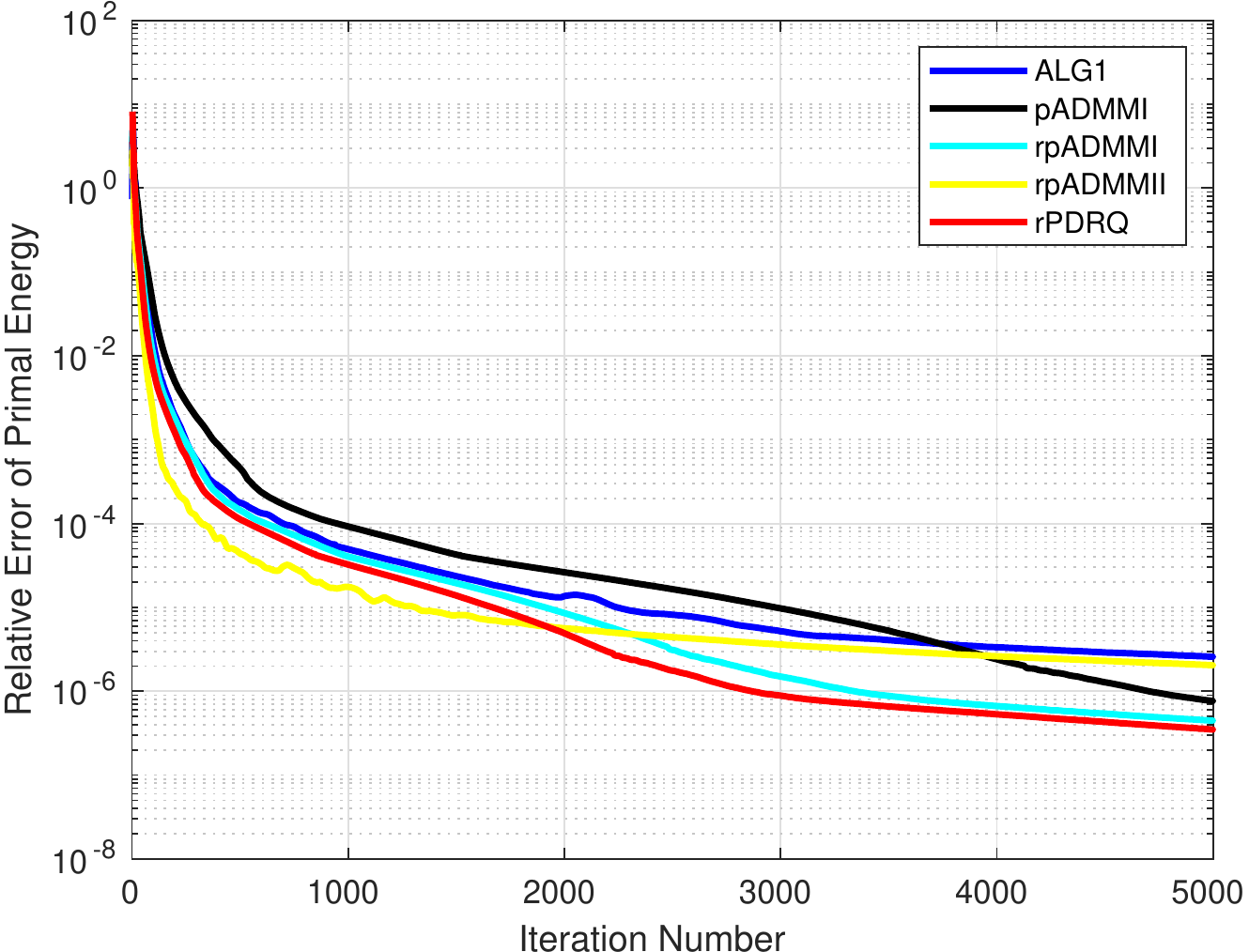}}\hfill
			\subfloat[Numerical convergence rate of relative primal energy gap
			compared with iteration time.]
			{\includegraphics[width=5.8cm]{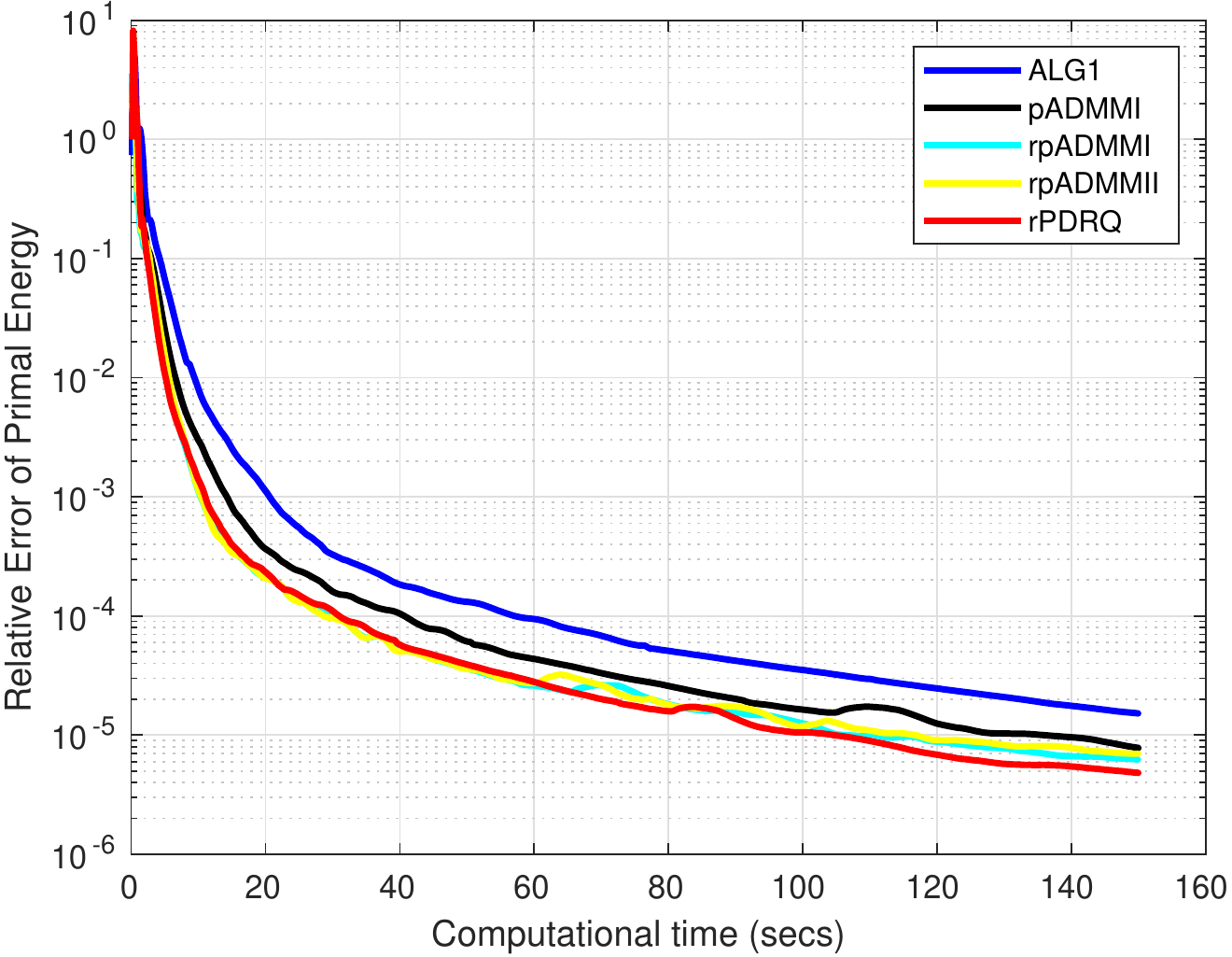}}
		\end{center}
		\caption{$\TV$-regularized image segmentation: numerical convergence rates.
			The relative error of primal energy is compared in terms of iteration
			number and computation time for
			Figure~\ref{tulips:denoise}(d)
			with  $\alpha = 0.5$. Note the 
			semi-logarithmic scale is used in the plot respectively.}
		\label{fig:tulips:multi}
	\end{figure}
	
	Table \ref{tab:2:labels} shows the results with over-relaxation and preconditioning. We can get faster and more efficient algorithms using rpADMMI, rpADMMII which have convergence guarantee, compared to pADMM-TY whose convergence can not be guaranteed. rpADMMI and rpADMMII are 30\% faster than pADMM-TY. It can also be seen that rPDRQ is slightly faster for the two-label segmentation case. Figure \ref{fig:shooter:performance} shows clearly that rpADMMI and rpADMMII are more efficient than pADMM-TY compared both with iteration numbers and computational time costs. Table \ref{tab:4:labels} and Figure \ref{fig:tulips:multi} show that rpADMMI and rpADMMII are very competitive for the 4-labeling case. They seem to be more robust to labels than rPDRQ. 
	
	Compared to \cite{Yuan2010,ybtb2014,YBTB10}, our observations from the numerical tests tell that  the proposed over-relaxed variants of preconditioned ADMM indeed bring out certain accelerations for both two-labeling and four-labeling cases. However, the final segmentation quality seems nearly the same after enough iterations with the same regularization parameters $\alpha$, which is probably due to the convexity of the segmentation models and robustness of the convex optimization algorithms.

	\section{Conclusions}\label{sec:conclusion}
	We give a systematic study on the augmented Lagrangian method for continuous maximal flow based image segmentation problems. We developed several novel and efficient preconditioned and over-relaxed ADMMs with convergence guarantee, together with relaxed and preconditioned Douglas-Rachford splitting method. Various efficient block preconditioners are proposed. 
	Numerical tests show that over-relaxed and preconditioned ADMM and Douglas-Rachford splitting methods have
	the potential to bring out appealing benefits and fast algorithms.

	\noindent
	
	\textbf{Acknowledgements}
	H. Sun acknowledges the support of NSF of China under grant No. \,11701563. He also acknowledges the support of Alexander
	von Humboldt Foundation during preparations of this work.
	
	\section{Appendix: Chambolle-Pock Splitting Algorithm} \label{sec:appendix}

	In this section, for completeness, we present the application  of the first-order primal-dual algorithm of \cite{CP} for solving \eqref{eq:two-label-mf} and \eqref{eq:potts-mf}. By the Fenchel's duality theory \cite{KK}, we can consider the equivalent primal-dual formulations for \eqref{eq:two-label-mf} and \eqref{eq:potts-mf} respectively:
	\begin{align}
	\min_u \max_{p_s, p_t, q}&\; \langle 1, p_s \rangle  -{I}_{\{p_s \leq C_s\}}(p_s) - {I}_{\{p_t \leq C_t\}}(p_t) - {I}_{\{\|q\|_{\infty} \leq \alpha\}}(q) \notag \\
	&\; + \langle u, p_t - p_s + \D q\rangle  \, ; \label{eq:pd:2:seg}
	\end{align}
	and 
	\begin{align}
	\min_{{\bf u}} \max_{{\bf u}, {\bf q}, {\bf p}}& \; \langle 1, p_s \rangle - \sum_{i=1}^n I_{\{p_i \leq \rho(l_i,x)\}}(p_i)  -\sum_{i=1}^n I_{\{\|q_i\| \leq \alpha\}}(q_i) \notag \\
	&\; + \sum_{i=1}^n \langle u_i, \D q_i +p_i -p_s \rangle \, . \label{eq:pd:m:seg}
	\end{align}
	%
	Both primal-dual models \eqref{eq:pd:2:seg} and \eqref{eq:pd:m:seg} generalized in a typical saddle-point optimization form \eqref{eq:saddle:qr},
	for which Chambolle and Pock proposed an efficient first-order primal-dual algorithm~\cite{CP}:
	\begin{align}
	y^{k+1}& = (I + \sigma \partial G )^{-1}( y^k + \sigma K  \bar x^k ), \notag \\
	x^{k+1}& = (I + \tau \partial F)^{-1}(x^k - \tau K^* y^{k+1} ), \label{alg:pd} \\
	\bar x^{k+1} &= 2x^{k+1} - x^k\, , \notag
	\end{align}
	and showed that $\sigma \tau \|{K}^* {K}\| < 1$ is required for the convergence of such iterative primal-dual scheme~\cite{CP}.
	
	For the primal-dual formulation \eqref{eq:pd:2:seg} of foreground-background image segmentation, we consider the saddle-point structure  \eqref{eq:saddle:qr} with the data \eqref{eq:pdrq:data:f} and \eqref{eq:pdrq:data:g}.
	Thus, as the proposed Chambolle-Pock primal-dual algorithm \eqref{alg:pd}, the primal-dual optimization problem \eqref{eq:pd:2:seg}
	can be directly solved by the following algorithm:
	\begin{equation}\label{iteration:pd:seg}
	\begin{cases}
	q^{k+1} = \mathcal{P}_{\alpha}(q^k - \sigma \nabla \bar u^k), \\
	p_{t}^{k+1} = \mathcal{P}_{C_t}(\bar p_{t}^k + \sigma \bar u^{k}),  \\
	p_{s}^{k+1} = \mathcal{P}_{C_s}(\bar p_{s}^k - \sigma \bar u^{k}),\\
	u^{k+1} =  u^k - \tau (\D q^{k+1} + p_{t}^{k+1} - p_{s}^{k+1}), \\
	\bar u^{k+1} = 2 u^{k+1} - u^k,
	\end{cases}
	\end{equation}
	where $\sigma \tau \|{K}^* {K}\| < 1$ with $K$ defined in \eqref{eq:pdrq:data:f} is required for convergence~\cite{CP}. In view of
	\[ 
	\|{K}^* {K}\|  \leq ( \|\nabla^* \nabla  \| + 2) < 10 \, ,
	\]
	we see that $\sigma \tau \leq 1/10$. 
	
	For the primal-dual model \eqref{eq:pd:m:seg} of multiphase image segmentation and the given notations \eqref{alm:multi:variables}, we study \eqref{eq:saddle:qr} with \eqref{eq:saddle:multi:oper:notion}, \eqref{eq:g:pd:form} and the same 
	linear operators $A$, $B$  as in \eqref{eq:AB:notion} and \eqref{eq:LI:notion}.
	Also, using the proposed Chambolle-Pock primal-dual algorithm \eqref{alg:pd},
	we have the Chambolle-Pock type primal-dual algorithm as follows:
	
	\begin{equation}\label{eq:alg1:4labels}
	\begin{cases}
	y^{k+1} = (I + \sigma \partial G )^{-1}( y^k + \sigma K  \bar u^k ), \\
	q_i^{k+1} = \mathcal{P}_{\alpha}(q_i^k - \sigma \nabla u_i^k), \quad i=1 ... n, \\
	p_i^{k+1}  = \mathcal{P}_{\rho(l_i,x)}(p_i^k + \sigma \bar u_i^k ), \quad i=1 ... n, \\
	p_s^{k+1} = p_s^k -\sigma \sum_{i=1}^n \bar u_i^k + \sigma,  \\
	u_i^{k+1} = u_i^k - \tau (\D q_i^{k+1} + p_i^{k+1} - p_s^{k+1}), \quad i=1 ... n, \\
	\bar u_i^{k+1} = 2u_i^{k+1} - u_i^k, \quad i=1 ... n,
	\end{cases}
	\end{equation}
	where the step sizes $\sigma$ and $\tau$ satisfy $\sigma \tau \|K^*K\| < 1$  with $K$ defined in \eqref{eq:saddle:multi:oper:notion}. With Lemma \ref{lem:dr:pre:multi}, we see 
	\[
	K^*K \leq \|\Delta + (n+1)I\| < 13.
	\]
	We thus choose $(9+n)\sigma \tau  \leq 1$.
	
	The first-order primal-dual is a popular first order algorithm for a lot of imaging applications. For numerical comparison, we present the above details.


\begin{thebibliography}{99}
		
		
		\bibitem{CP} A. Chambolle, T. Pock, \emph{A first-order primal-dual algorithm for convex problems with applications to imaging},  J. Math. Imaging and Vis., 2010, 
		40(1), pp. 120--145.
		
		\bibitem{BYT} E. Bae, J. Yuan, XC. Tai, \emph{Global minimization for continuous multiphase partitioning problems using a dual approach}, Int. J. Comput Vis., 92(1), pp. 112--129, 2011.
		
		
		\bibitem{Beck2009A}
		A. Beck, M. Teboulle,
		\newblock {\it A fast iterative shrinkage-thresholding algorithm for linear inverse
			problems},
		\newblock {SIAM Journal on Imaging Sciences}, 2(1):183--202, 2009.
		
		
		\bibitem{citeulike1859441}
		D.~P. Bertsekas,
		\newblock {Nonlinear Programming},
		\newblock {Athena Scientific}, September 1999.
		
		
		\bibitem{Boykov01fastapproximate}
		Y. Boykov, O. Veksler, R. Zabih,
		\newblock {\it  Fast approximate energy minimization via graph cuts}, 
		\newblock {IEEE Transactions on Pattern Analysis and Machine Intelligence},
		23:2001, 2001.
		
		\bibitem{Boykov01anexperimental}
		Y. Boykov, V. Kolmogorov,
		\newblock {\it An experimental comparison of min-cut/max-flow algorithms for energy
			minimization in vision},
		\newblock {IEEE Transactions on Pattern Analysis and Machine Intelligence},
		26:359--374, 2001.
		
		\bibitem{BS0} K. Bredies, H. Sun, \emph{A proximal point analysis of the preconditioned alternating direction method of multipliers}, J. Optim. Theory Appl. 173(3), pp. 878--907, 2017.
		
		\bibitem{BS} K. Bredies, H. Sun, \emph{Preconditioned Douglas-Rachford algorithms for TV and TGV regularized variational imaging problems}, J. Math. Imaging Vis., 2015, 52(3), pp. 317--344, doi {10.1007/s10851-015-0564-1}.
		
		
		\bibitem{BS1} K. Bredies, 	V. Horak, H. Sun, \emph{A unified analysis for relaxed and inertial variants of
			preconditioned Douglas–Rachford algorithms}, to appear, 2019. 
		
		
		\bibitem{citeulike3001108}
		A. Chambolle,
		\newblock {\it An algorithm for total variation minimization and applications},
		\newblock { J. Math. Imaging Vis.}, 20(1):89--97,
		January 2004.
		
		\bibitem{ChambolleP11}
		A. Chambolle, T. Pock,
		\newblock {\it A first-order primal-dual algorithm for convex problems with
			applications to imaging},
		\newblock {J. Math. Imaging Vis.}, 40(1):120--145,
		2011.
		
		
		\bibitem{Nikolova2006}
		Tony~F. Chan, S. Esedo{\=g}lu, M. Nikolova,
		\newblock {\em Algorithms for finding global minimizers of image segmentation and
			denoising models},
		\newblock {SIAM J. Appl. Math.}, 66(5):1632--1648, 2006.
		\bibitem{CBHY} C. Chen, B. He, Y. Ye, X. Yuan, \emph{The direct extension of ADMM for multi-block convex minimization problems is not necessarily convergent},  Math. Program.,  2016, 155(1-2), pp. 57--79.
		
		%
		
		\bibitem{DY}  W. Deng, W. Yin, \emph{On the global and linear convergence of the generalized alternating direction method of multipliers}, J. Sci. Comput., 66(3), 2016, pp. 889--916.
		
		
		
		\bibitem{EP} {J. Eckstein, D. P. Bertsekas}, {\it On the Douglas--Rachford splitting method and the proximal algorithm for maximal monotone operators}, Math. Program., 55, (1992), pp. 293--318.
		
		
		
		
		\bibitem{FG} {M. Fortin, R. Glowinski}, {\it On decomposition-coordination methods using an augmented Lagrangian}, in: M. Fortin
		and R. Glowinski, eds., { Augmented Lagrangian Methods: Applications to the Solution of Boundary Value Problems}, North-Holland, Amsterdam, 1983.
		
		\bibitem{goldstein2014fast}
		T. Goldstein, S. Osher,
		\newblock {\em The split bregman method for l1 regularized problems}, 
		\newblock {SIAM Journal on Imaging Sciences}, 2(2):323--343, 2009.
		
		\bibitem{He2002A}
		B. He, L. Liao, D. Han, H. Yang,
		\newblock {\em A new inexact alternating directions method for monotone variational
			inequalities},
		\newblock {Math. Program.}, 92(1):103--118, 2002.
		
		\bibitem{KK} {K. Ito, K. Kunisch}, { Lagrange Multiplier Approach to Variational Problems and Applications}, Advances in design and control 15, Philadelphia, SIAM, 2008.
		
		
		\bibitem{iwr_08}
		J. Lellmann, J. Kappes, J. Yuan, F. Becker, C.
		Schn\"{o}rr,
		\newblock Convex multi-class image labeling by simplex-constrained total
		variation.
		\newblock In {SSVM '09}, pp. 150--162, 2009.
		
		
		\bibitem{LST} {M. Li, D. Sun, K. C. Toh}, {\it A majorized ADMM with
			indefinite proximal terms for linearly constrained convex composite
			optimization}, SIAM J.~Optim. 26(2), 922--950, 2016.
		
		\bibitem{Nesterov2005Smooth}
		Yu. Nesterov,
		\newblock {\em Smooth minimization of non-smooth functions},
		\newblock {Math. Program.}, 103(1):127--152, 2005.
		
		\bibitem{Rockafellar1976}
		R.~T. Rockafellar,
		\newblock {\em Augmented {L}agrangians and applications of the proximal point
			algorithm in convex programming},
		\newblock { Math. Oper. Res.}, 1(2):97--116, 1976.
		
		
		\bibitem{rudin1992nonlinear}
		L.~Rudin, S.~Osher, E.~Fatemi,
		\newblock {\em Nonlinear total variation based noise removal algorithms}.
		\newblock { Physica D}, 60(1-4):259--268, 1992.
		
		\bibitem{SUN} H. Sun, \emph{Analysis of fully preconditioned alternating direction method of multipliers with relaxation in Hilbert spaces}, J. Optim. Theory Appl., 2019, pp. 1--31, https://doi.org/10.1007/s10957-019-01535-6. 
		
		
		
		

		
		
		\bibitem{Yuan2010}
		J. Yuan, E. Bae, XC. Tai,
		\newblock {\em A study on continuous max-flow and min-cut approaches,}
		\newblock In { IEEE Conference on Computer Vision and Pattern Recognition
			(CVPR), 2010}.
		
		\bibitem{yuan2007discrete}
		J. Yuan, C. Schn{\"o}rr, E. M{\'e}min,
		\newblock {\em Discrete orthogonal decomposition and variational fluid flow
			estimation}, 
		\newblock {J. Math. Imaging Vis.}, 28(1):67--80, 2007.
		
		
		
		
		
		
		
		
		
		
		
		
		
		
		
		
		
		\bibitem{ybtb2014}
		J. Yuan, E. Bae, XC. Tai, Y. Boykov,
		\newblock A spatially continuous max-flow and min-cut framework for binary
		labeling problems.
		\newblock { Numerische Mathematik}, 126(3):559--587, 2014.
		
		\bibitem{YBTB10}
		J. Yuan, E. Bae, XC. Tai, Y. Boykov,
		\newblock A continuous max-flow approach to potts model.
		\newblock In { ECCV}, 2010.
		
		\bibitem{zhang2010analysis}
		T. Zhang,
		\newblock Analysis of multi-stage convex relaxation for sparse regularization.
		\newblock {The Journal of Machine Learning Research}, 11:1081--1107, 2010.
		
		
		%
		%
		%
		%
		%
		%
		%
		%
		%
		%
		
		
		%
		%
		
	\end{thebibliography}
\end{document}